\documentclass[a4paper,10pt]{article}

\usepackage{amsmath, amssymb, amsthm}
\usepackage{graphics}

\usepackage{psfrag, booktabs}
\usepackage{color, epsfig, float}
\usepackage[active]{srcltx}
\usepackage[bookmarksopen, colorlinks, linkcolor = blue, urlcolor = red,
            citecolor = red, menucolor = blue]{hyperref}

\newtheorem{theorem}{Theorem}

\newcommand\om{\omega}

\newcommand\F{\mathrm{F}}

\newcommand\R{\mathbb{R}}

\newcommand\tm{\subseteq}

\newcommand\norm[1]{\|#1\|}
\newcommand\set[1]{\{#1\}}
\newcommand\abs[1]{|#1|}
\newcommand\f {\mathbf}

\newcommand\dd{\mathrm{d}}

\newcommand\Mo{\mathrm{M}}
\newcommand\Ro{\mathrm{R}}

\newcommand\di{\partial}

\newcommand\req[1]{(\ref{eq:#1})}

\newcommand\bracket[1]{\langle#1\rangle}

\newcommand\fs {\boldsymbol}

\newcommand{\x}{x}
\newcommand{\y}{y}

\title
{Regularization of systems of nonlinear ill-posed equations:
II. Applications}

\author
{}

\begin{document}





\maketitle

\centerline{\scshape Markus Haltmeier}
{\footnotesize
\centerline{Department of Computer Science, University of Innsbruck}
\centerline{Technikerstr. 21a, A-6020 Innsbruck, Austria}
}
\medskip
\centerline{\scshape Richard Kowar}
{\footnotesize
\centerline{Department of Computer Science, University of Innsbruck}
\centerline{Technikerstr. 21a, A-6020 Innsbruck, Austria}
}
\medskip
\centerline{\scshape Antonio Leit\~ao}
{\footnotesize
\centerline{Department of Mathematics, Federal University of St. Catarina}
\centerline{P.O. Box 476, 88040-900 Florian\'opolis, Brazil}
}
\medskip
\centerline{\scshape Otmar Scherzer} {\footnotesize
\centerline{Department of Computer Science, University of
Innsbruck} \centerline{Technikerstr. 21a, A-6020 Innsbruck,
Austria,} \centerline{ and } \centerline{RICAM, Austrian Academy
of Sciences} \centerline{ Altenberger Stra\ss{}e 69, A-4040 Linz,
Austria.} }
\medskip


\begin{abstract}
In part I we introduced modified Landweber--Kaczmarz methods and
have established a convergence analysis. In the present work we
investigate three applications: an inverse problem related to thermoacoustic
tomography, a nonlinear inverse problem for semiconductor equations, and a
nonlinear problem in Schlieren tomography. Each application is
considered in the framework established in the previous part. The
novel algorithms show robustness, stability, computational
efficiency and high accuracy.
\end{abstract}

\section{Introduction}
\label{sec:intro}

In \cite{HLS06} we analyzed novel iterative
\emph{Landweber--Kaczmarz methods} for solving \textit{systems of
ill-posed equations}
\begin{equation}\label{eq:u-ix}
    \F_i( \x ) = \y^{\delta,i}\,,  \quad \;  i = 0, \dots,  N-1  \;.
\end{equation}
Here $\F_i: D_i\tm X \to Y$ are operators between Hilbert spaces $X$
and $Y$ and $\y^{\delta,i}$ are approximations of noise free data
$\y^i$ with $\norm{\y^{\delta,i} - \y^i} \leq \delta^i $ for $i \in
\set{ 0, \dots,  N-1}$.

The first method  analyzed in \cite{HLS06} is the \textit{loping}
Landweber--Kaczmarz (\textsc{lLK}) method
\begin{equation} \label{eq:lwk-lop}
    \x_{n+1} = \x_{n} - \om_n {\F_{[n]}'(\x_{n})}^\ast
                                ( \F_{[n]}(\x_{n}) - \y^{\delta,[n]}) \,,
\end{equation}
with a \emph{bang--bang relaxation parameter}
\begin{equation} \label{eq:skip}
    \om_n := \om_n( \delta^i,  \y^{\delta,i})   =
    \begin{cases}
        1  & \norm{ \F_{[n]}(\x_{n}) - \y^{\delta,i}}
        > \tau \delta^i \\
        0  & \mbox{ otherwise}
    \end{cases}\,.
\end{equation}
Here $\tau > 2$ is an appropriately chosen constant and $[n] := n
\mod N$. We refer to $N$ subsequent iterations as one iteration
\textit{cycle}. The \textsc{lLK} method skips inner iterations, if
the data are sufficiently well approximated. The whole iteration
is terminated if $\norm{ \F_i(\x_{n_\ast^\delta+i}) -
\y^{\delta,i}} < \tau \delta^i$ for \textit{all}  $i \in \set{ 0,
\dots, N-1}$.

The second method  analyzed in \cite{HLS06}, is the
\textit{embedded} Landweber--Kaczmarz (\textsc{eLK}) method,
\begin{eqnarray} \label{eq:lw-emb}
  \f \x_{n+1/2} &=
  \f \x_{n} - \omega_{n} \f F'(\f \x_{n})^\ast ( \f F(\f \x_{n})   - \f \y^\delta )\,,
  \\ \label{eq:lw-emb2}
  \f \x_{n+1} &=
  \f \x_{n+1/2} - \omega_{n+1/2} \f G ( \f \x_{n+1/2} ) \,.
\end{eqnarray}
Here $\f \x := (\x^i)_i \in X^N$, $\f \y^\delta := (\y^{\delta,i})_i \in Y^N$,
$\f F (\f \x)  :=  (\F_i(\x^i))_i \in Y^N$,
$\delta := \max\set{\delta^i: i = 0,\dots, N-1 }$ and $\f G$ is a scalar multiple
of the steepest descent direction of the functional
$ \mathcal{G} (\f \x) := \sum_{i=0}^{N-1} \norm{\x^{i+1} - \x^{i}}^2$
on $X^N$. Moreover,
\begin{eqnarray*}
\begin{aligned}
    \omega_n   &=
    \begin{cases}
        1  & \norm{ \f F(\f \x_n) - \f \y^{\delta}}
         > \tau \delta \\
        0  & \mbox{ otherwise }
    \end{cases}\,,
        \\
    \omega_{n+1/2}
    & =  \begin{cases}
      1  & \norm{ \f G ( \f \x_{n+1/2} ) }
      > \tau  \epsilon(\delta)\\
      0  & \mbox{ otherwise }
\end{cases}\,,
\end{aligned}
\end{eqnarray*}
with  $\epsilon(\delta) \to 0$, as $\delta \to 0$.
The \textsc{eLK} iteration is terminated after $n_\star^\delta$ iterations when
$\norm{ \f F ( \f \x_{n_\star^\delta} )  - \f \y^{\delta}} \leq \tau \delta$
and $\norm{ \f G( \f \x_{n_\star^\delta} ) }\leq \tau \epsilon(\delta)$.
One  \textit{cycle} of the \textsc{eLK} method is defined by performing both iterations
in~\req{lw-emb} and~\req{lw-emb2}. Motivations for these iteration methods can
be found in~\cite{HLS06}.

In~\cite{HLS06} we have proven stability and convergence for both the \textsc{eLK} and \textsc{lLK} method.
In this article we apply the methods to three different problems: a linear inverse problem related to
thermoacoustic tomography, an inverse problem for semiconductors and Schlieren tomography.
For validation, the results are compared to the classical \emph{Landweber--Kaczmarz} (LK) method
\cite{Nat97,KowSch02}
\begin{equation}\label{eq:lwk-class}
    \x_{n+1} =
    \x_{n} -
     \F_{[n]}^\prime(\x_{n})^\ast ( \F_{[n]}  (\x_{n}) - \y^{\delta,[n]} ) \,.
\end{equation}
The outline of this article is as follows. In Section \ref{sec:tct} we apply both methods,
the \textsc{lLK} method and the \textsc{eLK} method, to an inverse problem related to
\emph{thermoacoustic computed tomography}, which mathematically can be reduced to the inversion
of the circular mean transform.
In Section \ref{sec:semicond} we consider the exponentially ill-posed
nonlinear \emph{inverse problem of estimating the doping profile in a
semiconductor} from voltage--current measurements.
In Section \ref{sec:sch} we consider \emph{Schlieren tomography},
an imaging method for reconstructing three dimensional pressure
fields.

\section{Thermoacoustic computed tomography}
\label{sec:tct}

\textit{Thermoacoustic} computed tomography is a  promising new
imaging modality that has the potential to become a mayor
non--invasive imaging method.
In thermoacoustic imaging an object of interest is illuminated by short electromagnetic pulses,  such as optical illumination or radio waves,
which results in an excitation of acoustic waves (pressure variations).
The spatially varying initial  pressure distribution inside the object caries valuable structural
and functional information  and  is reconstructed from acoustical data which is recorded with detectors
outside of the object.
Among the several publications on thermoacoustic imaging we mention
\cite{AndEtMul00, KruKisReiKruMil03, WanPanKuXieStoWan03, XuMWan03,XuMWan06}.


\begin{figure}
    \begin{center}
        \includegraphics[height = 0.22\textwidth]{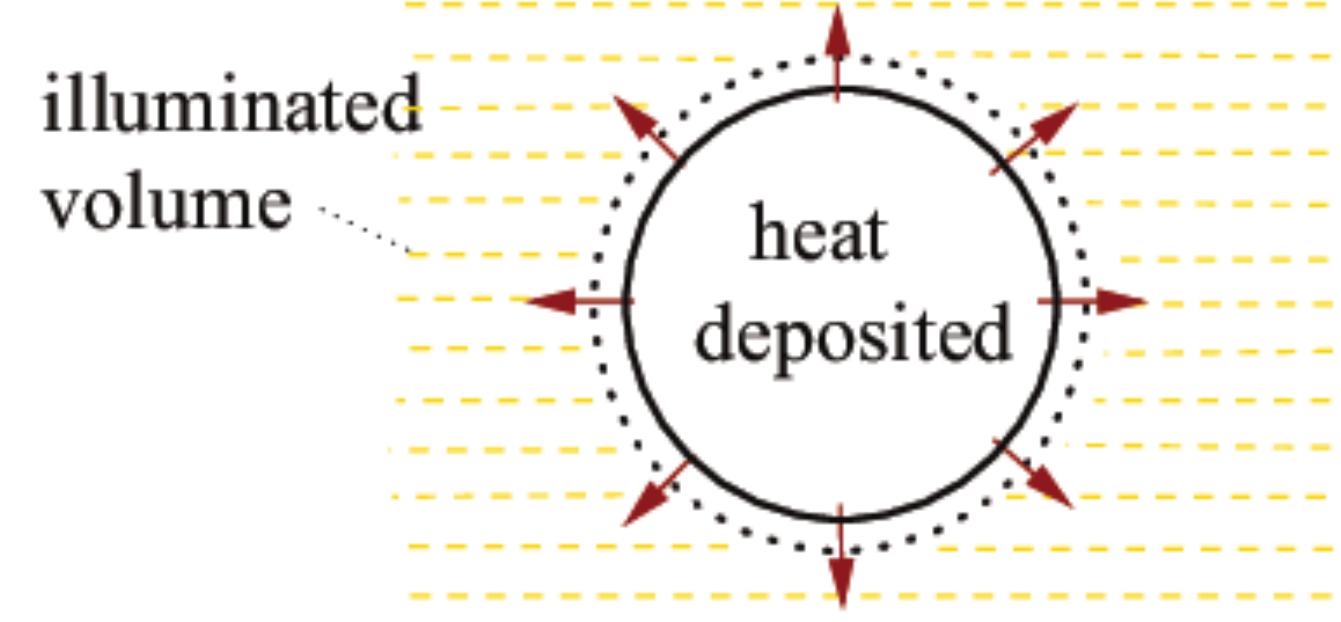}
        \hfill
        \includegraphics[height = 0.22\textwidth]{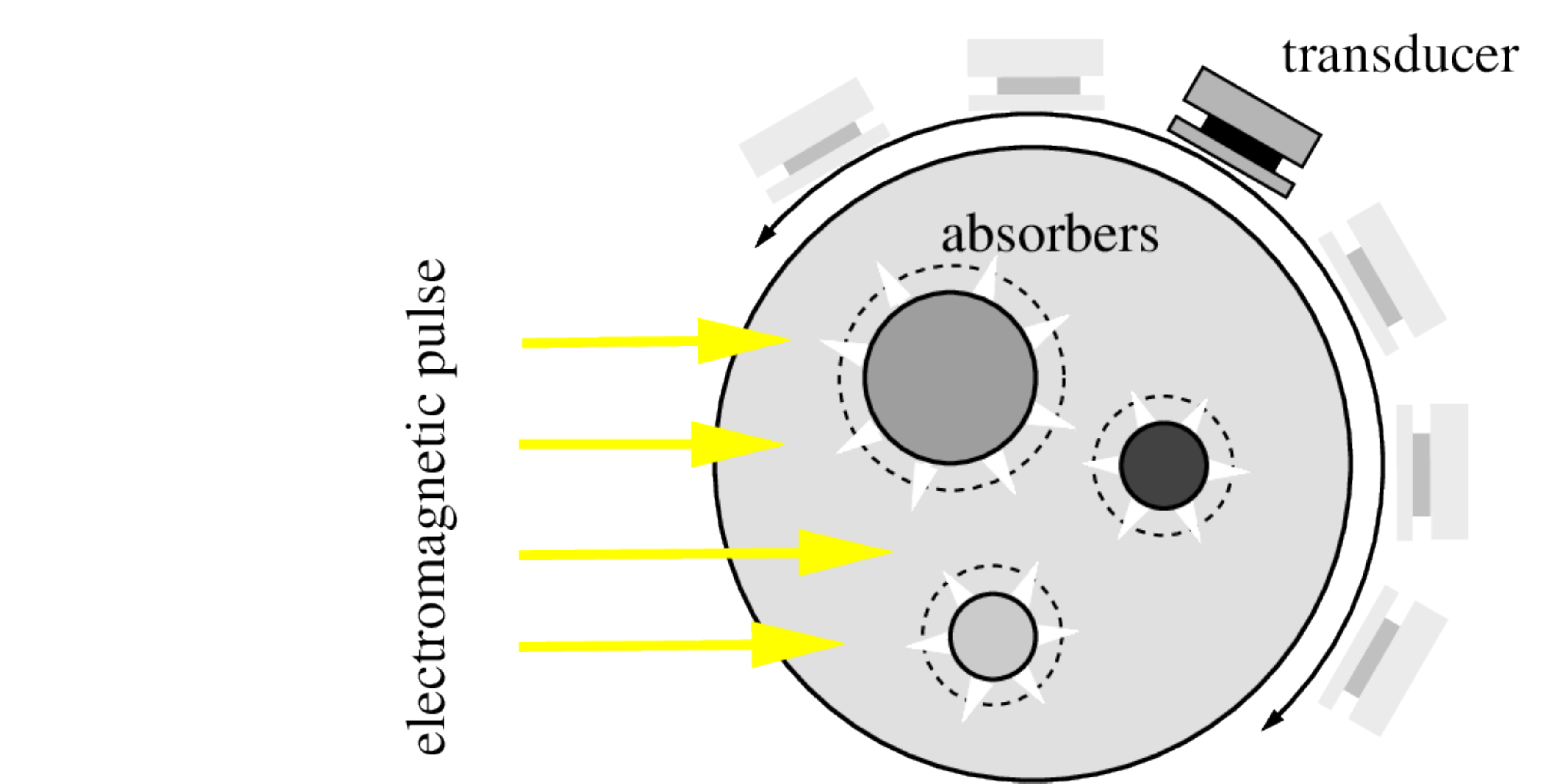}
    \end{center}
      \caption[Illustration of Thermoacoustic computed tomography]{{\bf Thermoacoustic computed tomography.}
        Left.  A sample is illuminated with pulsed electromagnetic energy.
        Right. The induced acoustic waves are recorded with  acoustic detectors
        located outside of the object.}
    \label{fg:tctscan}
\end{figure}

\subsection{Mathematical model}

Assume that the initial pressure distribution   $x ( \xi_1, \xi_2) $
is independent of the third spatial coordinate, let $\fs \xi
:= (\xi_1, \xi_2)$ and let $B_R$ denote the open disc (in $\R^2$)
with radius $R$.

Mathematically,  in such a situation, thermoacoustic tomography reduces to the
problem of recovering $x$ from \cite{HalSchBurNusPal06}
\begin{equation}\label{eq:eqMo}
    \Mo_i (\x  ) = \y^{\delta, i} \,, \quad i=0, \dots, N-1.
\end{equation}
with noisy data $\y^{\delta, i}$, where
\begin{equation} \label{eq:smean}
    \Mo_i (\x)(t)
   : =
    \frac{1}{\sqrt{\pi}}
        \int_{S^1}
        \x(\fs \xi_{i} + t \fs\sigma)
        \, \dd \Omega(\fs\sigma)
        \,,  \quad i = 0, \dots, N-1
\end{equation}
is the \textit{scaled mean value} of $\x  \in C_0^\infty(B_R)$ over the circle with center
$\fs \xi_i \in \di B_R$  and radius $t \in [0, 2R] $.
Here $S^1$ denotes the surface of the unit disc in $\R^2$ and $\dd \Omega$ the line measure
on $S^1$ and $(\fs \xi, 0)$ corresponds to the locations of an acoustic receiver, see Figure \ref{fg:tctscan}.
This section is devoted to the stable solution of \req{eqMo}.

Recently, we proposed a detection technique \cite{BurHofPalHalSch05} 
where an array of parallel \textit{line detectors} is rotated around the around the object 
and pressure signals along the line detectors are recorded. Then the solution of \req{eqMo} 
allows for reconstructing of a fully three dimensional initial pressure distribution.
In experiments line detectors are realized by thin laser beams that are either part of a 
Fabry--Perot or a Mach--Zehnder interferometer \cite{BurHofPalHalSch05, PalNusHalBur06}.

\subsection{Abstract formulation in Hilbert--spaces}

In the following we show that the problem of recovering of a function $x$
from its circular means can be put in the abstract framework of \cite{HLS06}.

In the remainder of this section, let $L^2(B_R)$ be the Hilbert space of square integrable functions on
$B_R$ with $\norm{x}^2 :=  \int_{B_R} x(\fs \xi)^2 d \fs \xi$. Moreover, we denote by $L^2([0, 2R], t \dd t )$ the
Hilbert space of all functions $\y: [0,\infty) \to \R$ (observable quantities) with support in $[ 0, 2R ]$ with
\begin{equation*}
    \norm{\y}_{Y}^2
    :=
        \int_0^\infty
            \y(t)^2 \,
        t\dd t
    < \infty \,.
\end{equation*}
Finally we denote by
\begin{equation*}
    \bracket{\y_1, \y_2}_{Y} =
    \int_0^\infty \y_1(t) \y_2(t)
    \, t \dd t
\end{equation*}
the associated inner product on $L^2([0,2R],t\dd t)$.

The scaled \textit{circular mean operators}
are defined by
\begin{equation*}
    \Mo_i: C_0^\infty(B_R)\to L^2([0,2R],t\dd t): \x \mapsto \Mo_i( \x )\,,
\end{equation*}
where $\Mo_i(\x)$ is defined
by \req{smean}.

\begin{theorem}
The operators $\Mo_i$ can be continuously extended to
$\Mo_i: L^2(B_R) \to L^2([0,2R],t\dd t)$ with $\norm{\Mo_i} \leq 1$.
The adjoint $\Mo_i^\ast: L^2([0,2R],t\dd t) \to L^2(B_R)$ is
given by
\begin{equation}\label{eq:sphM-ad}
    \Mo_i^\ast (\y) ( \fs \xi )  =
    \y(\abs{ \fs \xi_i - \fs \xi})/ \sqrt \pi
    \,,
    \quad \;  i = 0, \dots, N-1\;.
\end{equation}
\end{theorem}

\begin{proof}
Assume that $\x\in C_0^\infty( B_R )$. From the definition of $\Mo_i(\x)$,
the Cauchy Schwarz inequality and Fubini's theorem it follows that
\begin{equation*}
\begin{aligned}
    \norm{
            \Mo_i ( \x )
         }_{Y}^2
    &=
    \frac{1}{\pi} \int_{0}^\infty
    \left(
        \int_{S^1}
            \x( \fs \xi_i + t \fs\sigma ) \cdot \chi_{B_R}(\fs \xi_i + t \fs\sigma)
        \, \dd \Omega(\fs\sigma)
    \right)^2
    t\dd t
    \\
    &\leq
    \frac{1}{\pi}
    \int_{0}^\infty
                        \left( \int_{S^1}
            \chi_{B_R}(\fs \xi_i + t \fs\sigma)
        \, \dd \Omega(\fs\sigma)
        \int_{S^1}
                \x( \fs \xi_i + t \fs\sigma )^2
                \, \dd \Omega(\fs\sigma)
        \right)
        t\dd t
    \\
    &\leq
            \int_{0}^\infty
                \int_{S^1}
                    \x( \fs \xi_i + t \fs\sigma )^2
            \ \dd \Omega(\fs\sigma) \ t\dd t
    =
    \norm{ \x }^2\,.
\end{aligned}
\end{equation*}
Hence $\Mo_i$ is bounded from $C_0^\infty(\Omega)$ into $L^2([0,2R],tdt)$ and therefore can be extended to a bounded linear operator mapping on
$L^2(B_R)$ with $\norm{\Mo_i} \leq 1$.
In particular $\Mo_i$ has a bounded adjoint.

Let $\x \in C_0^\infty( B_R )$ and $\y \in C_0^\infty([0,2R])$. From Fubini's theorem
and the substitution $\fs \xi = \fs \xi_i + t \fs\sigma$ it follows that
\begin{equation*}
\begin{aligned}
    \bracket{
        \Mo_i ( \x ), \y
    }_{Y}
    & =
    \frac{1}{\sqrt \pi}
    \int_{0}^\infty
            \y(t)
            \left(
                \int_{S^1}
                \x( \fs \xi_i + t \fs\sigma )
                \, \dd \Omega(\fs\sigma)
            \right)
    t \dd t
    \\
    &=
    \frac{1}{\sqrt \pi} \int_{S^1}
        \int_{0}^\infty
            \x( \fs \xi_i + t \fs\sigma )
              \y(t)
         \, t\dd t \,
    \dd \Omega(\fs\sigma)
    \\
    &=
    \int_{B_R}
            \x( \fs \xi )
            \frac{\y(\abs{\fs \xi - \fs \xi_i})}{\sqrt \pi}
    \, \dd \fs \xi
    =
    \bracket{ \x , \Mo_i^\ast ( \y ) } \,.
\end{aligned}
\end{equation*}
This shows \req{sphM-ad}.
\end{proof}

We note that, for linear bounded operators, the
\emph{tangential cone condition}~\cite[Eq. (15)]{HLS06}  is  satisfied with $\eta =0$.
Since $\Mo_i$ is linear and $\norm{\Mo_i} \leq 1$, the \textsc{lLK} method and
the \textsc{eLK} method provide  convergent regularization methods for
solving \req{eqMo}.

\subsection{Numerical reconstruction}

The \textsc{lLK} method applied to thermoacoustic computed tomography reads as
\begin{equation} \label{eq:lwkM}
    \x_{n+1} =
    \x_{n} -
    \om_n  \Mo_{[n]} ^\ast (\Mo_{[n]} (\x_{n}) - \y^{\delta,[n]}) \,,
\end{equation}
with
\begin{equation} \label{eq:skipM}
\om_n   = \om_n( \delta,  \y^\delta)   =
\begin{cases}
      1  & \norm{ \Mo_{[n]}(\x)  - \y^{[n], \delta}}
      > \tau \delta^{[n]}\\
      0  & \mbox{ otherwise }
\end{cases}\,.
\end{equation}

\begin{figure}
\begin{center}
\includegraphics[width = 0.48\textwidth, height = 0.4 \textwidth]{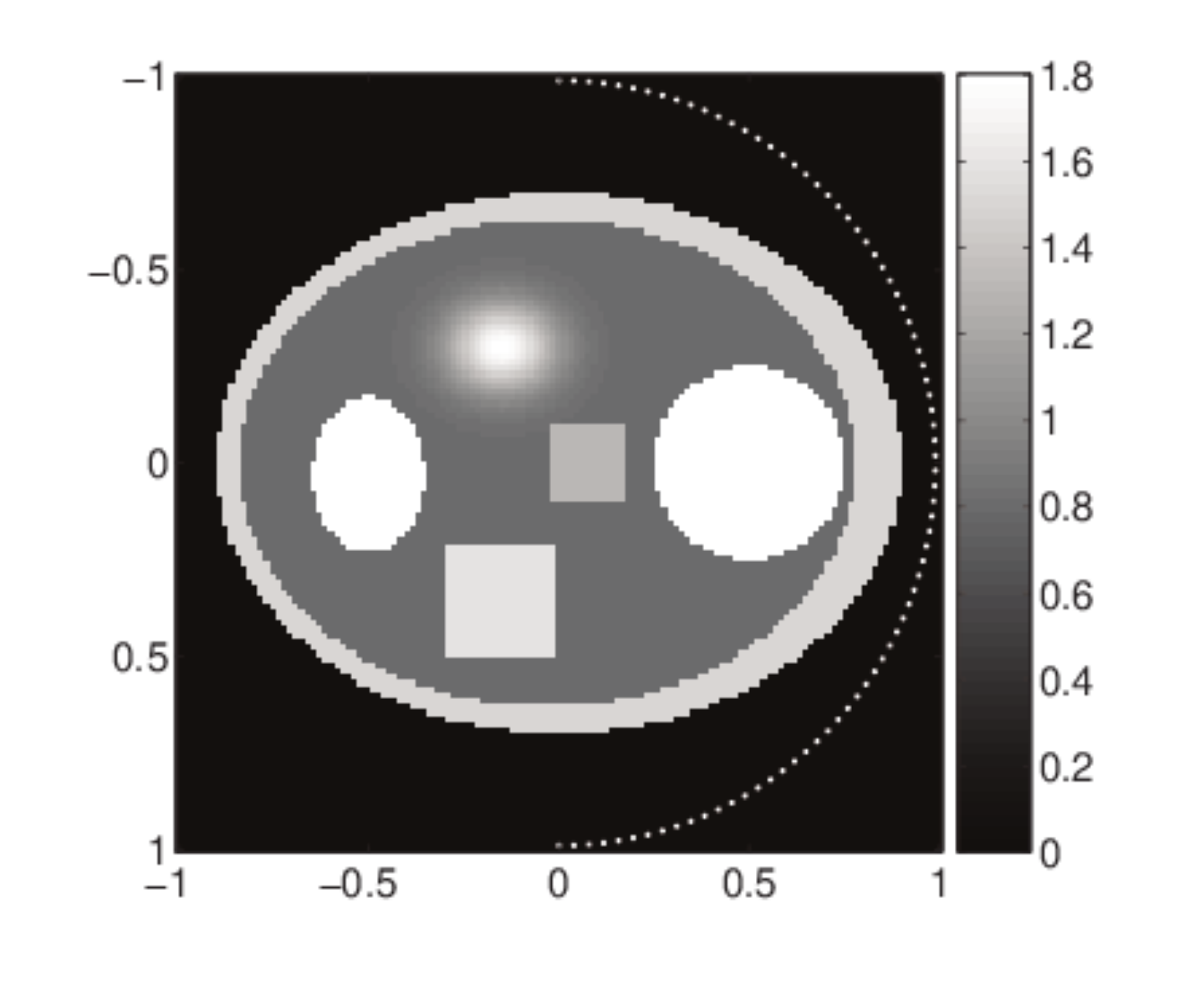}
\includegraphics[width = 0.48\textwidth, height = 0.4 \textwidth]{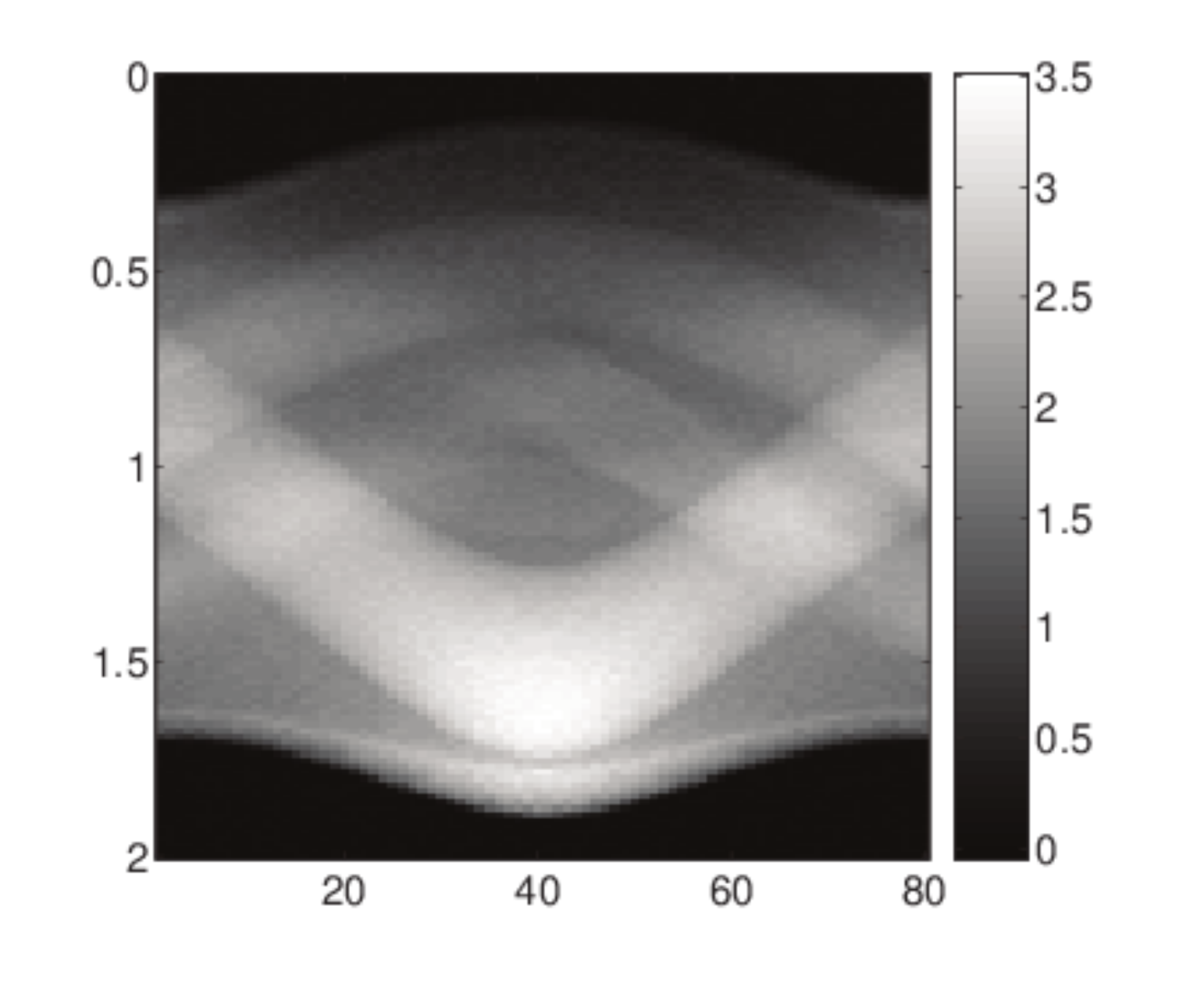}
\caption{The left picture shows the phantom to be reconstructed, where the white dots
indicate the locations of the detectors. The corresponding data
$(\Mo_i(\x))_i$ is depicted in the right image.}
\label{fg:tat-original}
\end{center}
\end{figure}

In the numerical implementation the spaces $L^2(B_R)$ and $L^2( [0,2R], t\dd t)$
are approximated by the linear spans of piecewise linear splines. Each spline is represented
by vectors in $\R^{M \times M}$ and $\R^M$, respectively.
For the numerical evaluation of $\Mo_i$ the integration over $S^1$ in (\ref{eq:smean}) is performed
with the trapezoidal quadrature formula.
This requires $\mathcal O(M)$ floating point operations (FLOPS) for any radius.
Therefore, one iteration in the \textsc{eLK} method according to \req{lwkM}, \req{skipM}
requires $\mathcal O(M^2)$ FLOPS.  In thermoacoustic tomography  $M = \mathcal  O(N)$, and the
numerical effort for performing a complete \textit{Landweber--Kaczmarz cycle}
is $\mathcal  O(N^3)$. This is the same complexity that is needed to perform one
step in the Landweber iteration.

\begin{figure}
\begin{center}
\includegraphics[width = 0.48\textwidth, height = 0.4 \textwidth]{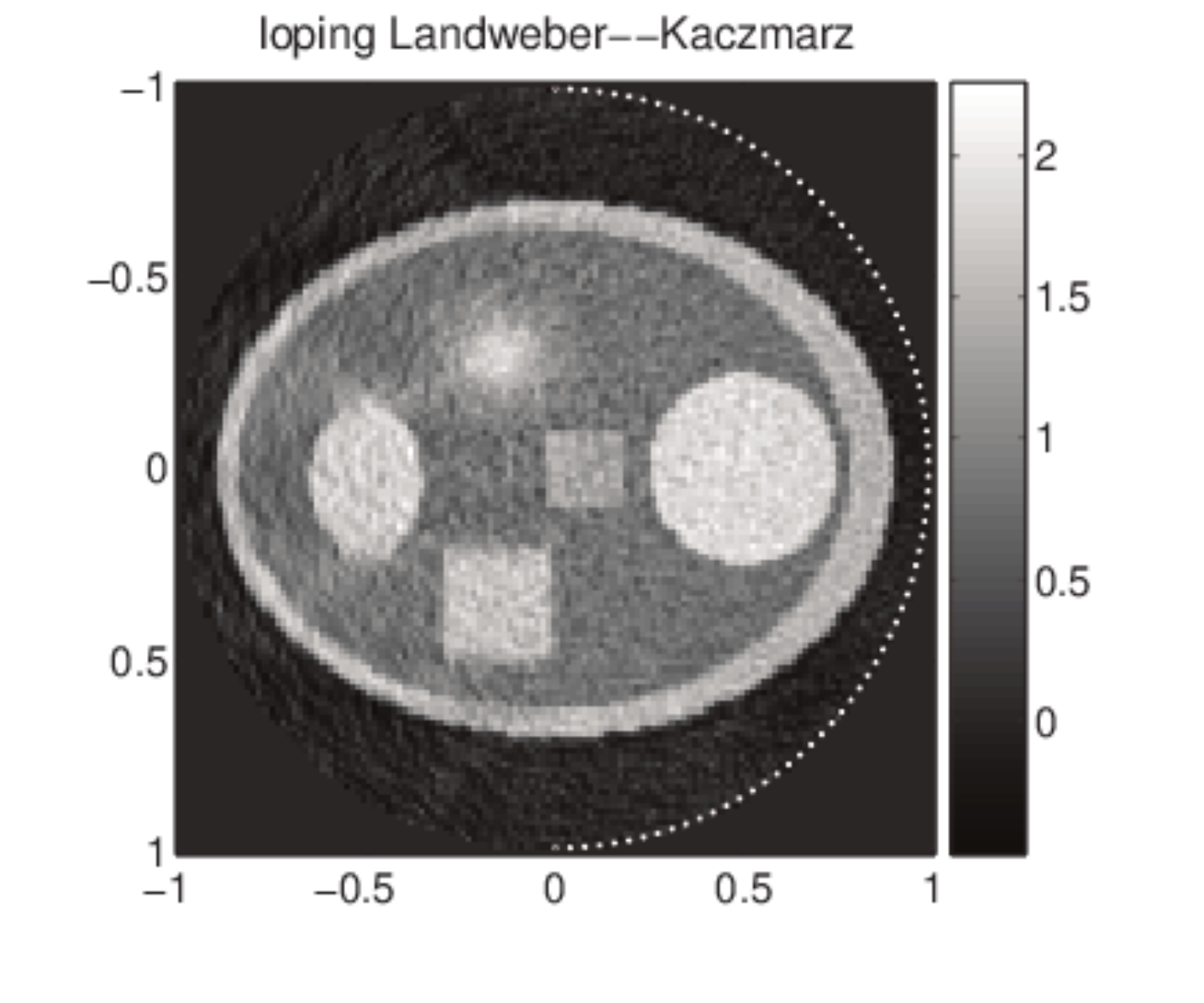}
\includegraphics[width = 0.48\textwidth, height = 0.4 \textwidth]{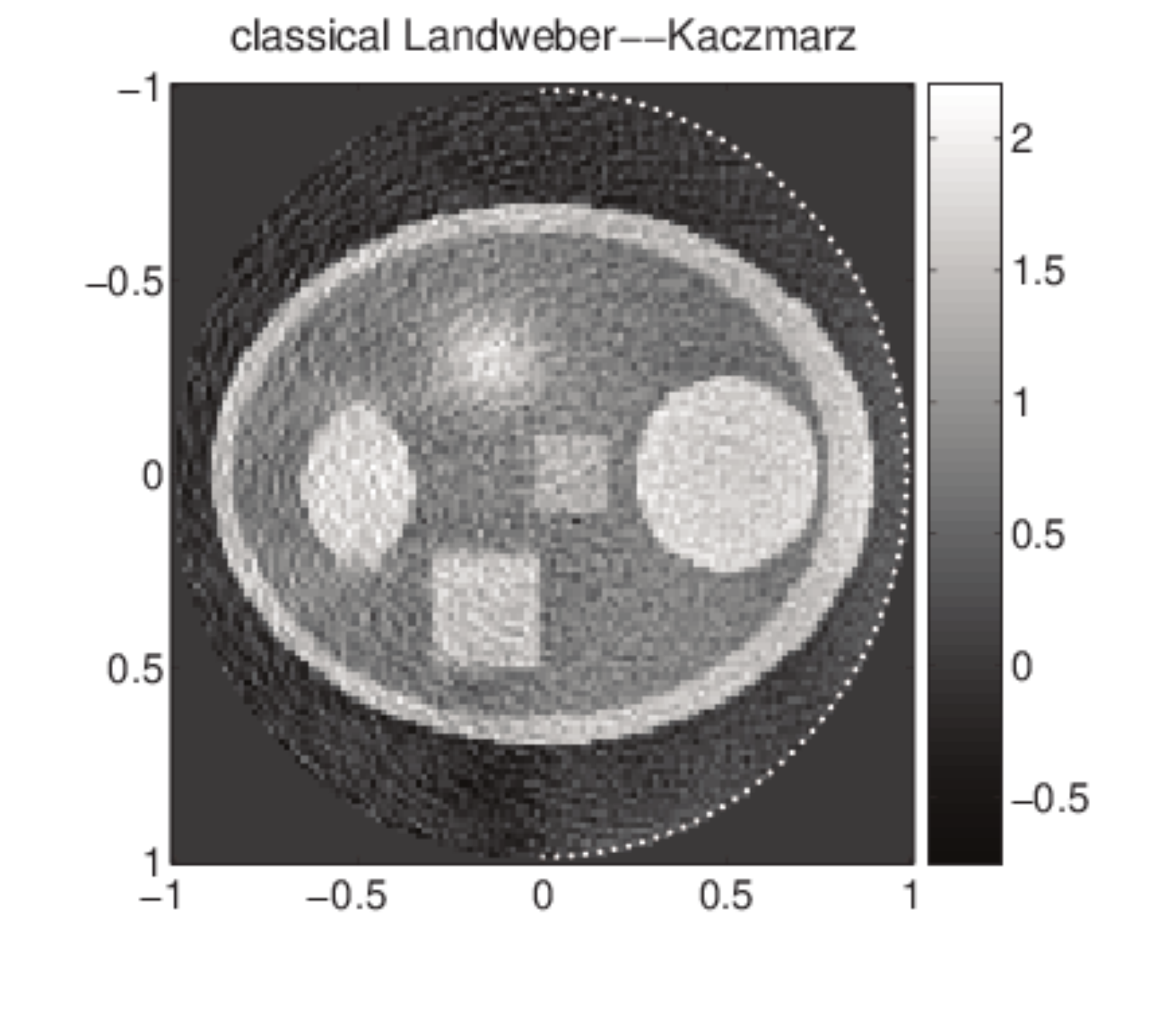}
\caption{Numerical reconstruction of the phantom depicted in Figure \ref{fg:tat-original}.
The left picture shows the regularized solution $\x_{n^\ast_\delta}$
with the \textsc{lLK} method \req{lwkM}, \req{skipM}. The right picture shows
the reconstruction with the \textsc{LK} method using the same number of
$n^\ast_\delta$ of iterations.}
\label{fg:tat-recon}
\end{center}
\end{figure}

\begin{figure}
\begin{center}
\includegraphics[height = 0.3 \textwidth]{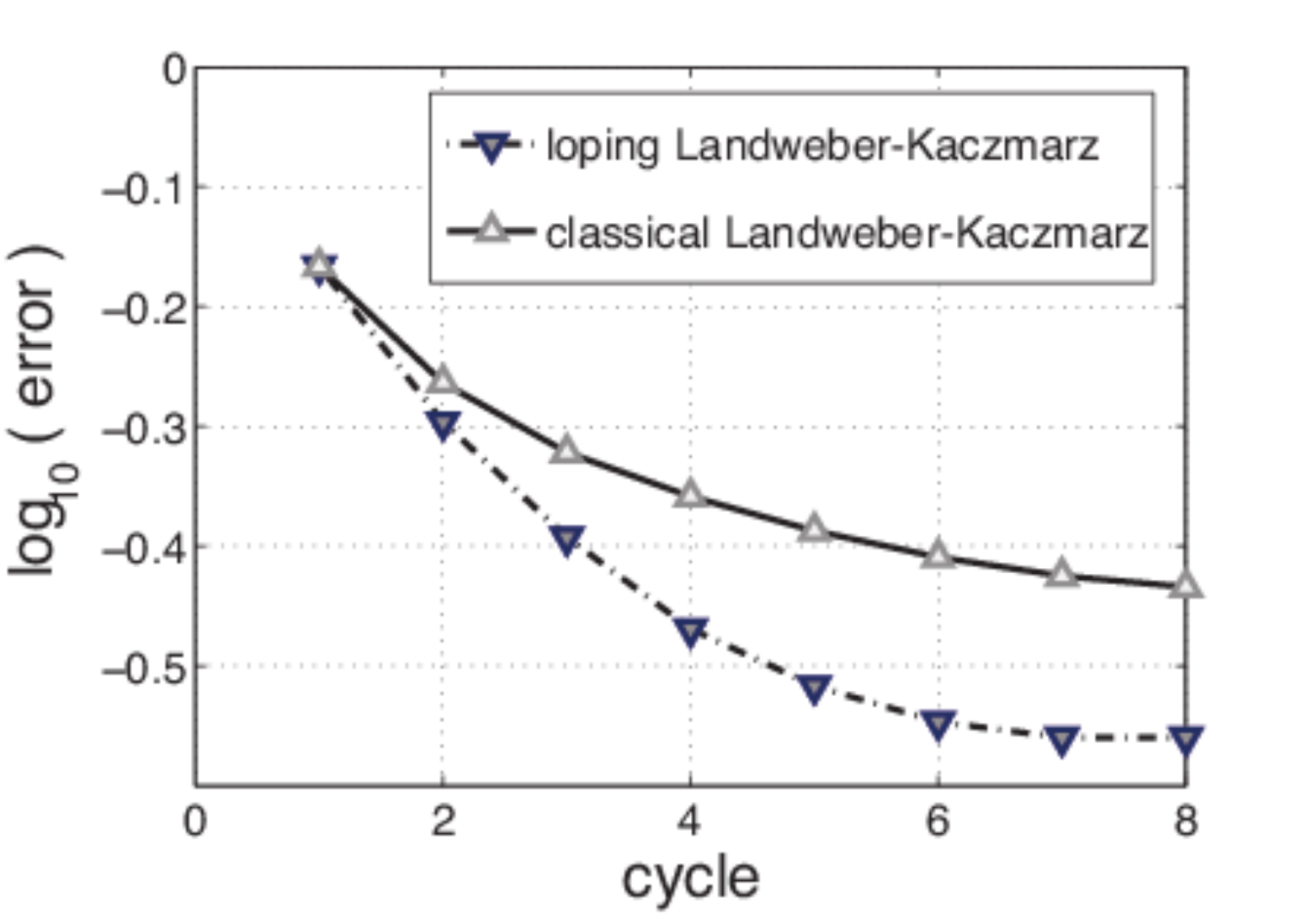}
\includegraphics[height = 0.3 \textwidth]{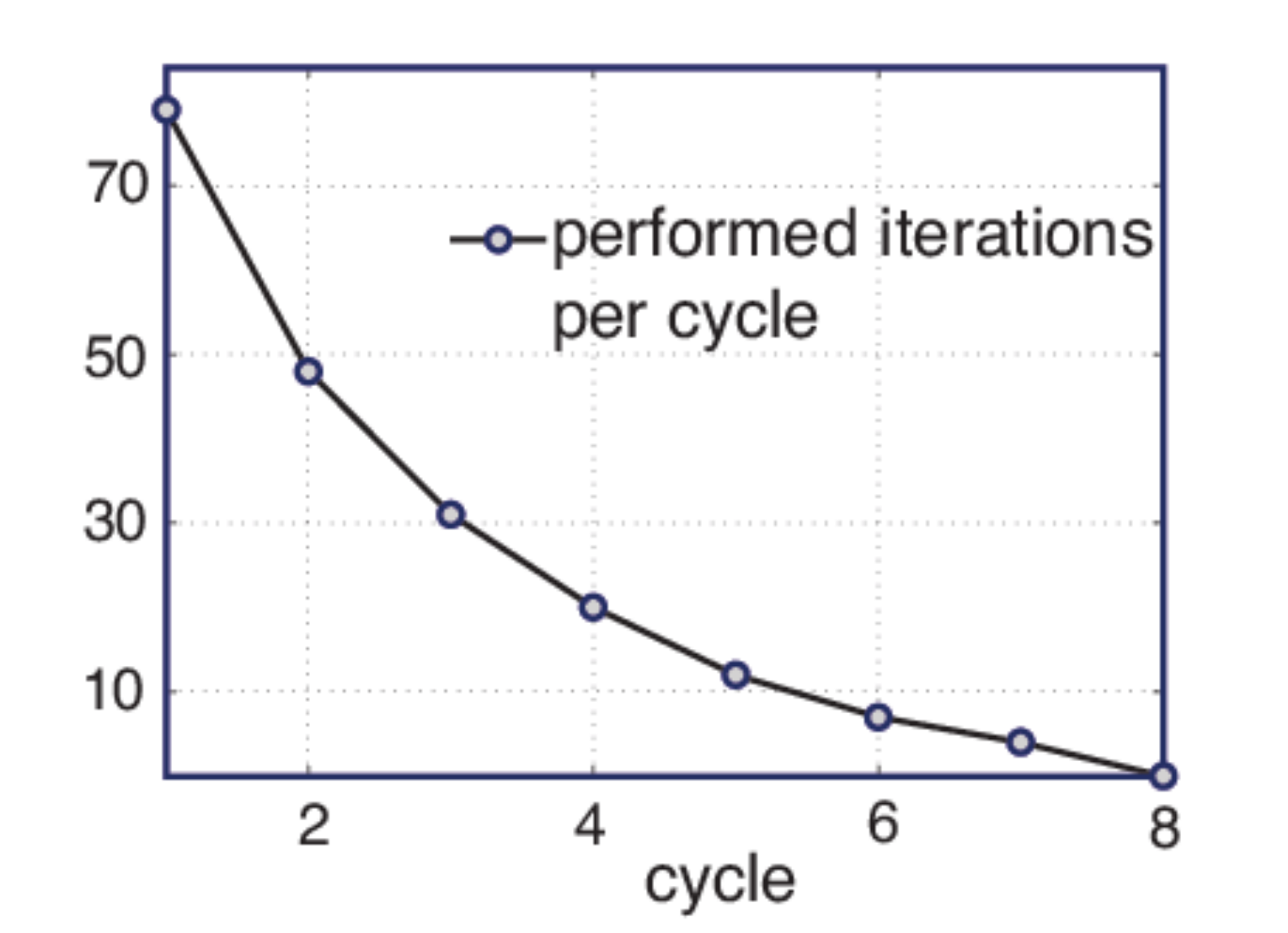}
\caption{The left image shows the decrease of the error using
the \textsc{lLK} method \req{lwkM}, \req{skipM} and the \textsc{LK} method. 
The right picture shows the actually performed number of iterations within a 
cycle of the \textsc{lLK} method. Within the $8$--th cycle all $\omega_{n} = 0$ 
and the iteration terminates.}
\label{fg:tat-steps}
\end{center}
\end{figure}

In the following numerical examples we consider $N=80$ measurements, where
the centers $\xi_i = R( \sin(\pi i /N) \cos(\pi i/N) )$  are uniformly 
distributed on the semicircle 
$S^+ := \set{ (\xi_1, \xi_2) \in \partial B_R:  \xi_1 \geq 0}$.
The phantom  $x$,  shown in the right picture in Figure  \ref{fg:tat-original}, 
consists  of a superposition of characteristic functions and one Gaussian kernel. 
The data $\Mo_i(\x)$ was calculated via numerical integration and  $5\%$ 
uniformly distributed noise was added. 
Micro-local analysis predicts, that in such  limited angle situation certain 
details of $x$ outside the convex hull of $S^+$  cannot be  recovered in a   
stable way~\cite{LoiQui00, XuYEtAl04}. 
The numerical reconstruction with the  \textsc{lLK} method, implemented with $\tau = 2.0$, is depicted in Figure  \ref{fg:tat-recon}.   The stopping rule becomes active after 8 cycles.
For comparison purposes, the result of the \textsc{LK} method is plotted after 8 cycles.
Figure \ref{fg:tat-embed}  shows results obtained with the \textsc{eLK} method using the same phantom. We remind that the  iterates of the \textsc{eLK} method  are vectors $\mathbf{\x}_n=(\x^0_n, \dots \x^{N-1}_n)$. For the results depicted in Figure \ref{fg:tat-embed} we used the average
$\x_{n} := \sum_{i=0}^{N-1} \x^{i}_n$ over all components.

\begin{figure}
\begin{center}
\includegraphics[height = 0.4 \textwidth, width = 0.48\textwidth]{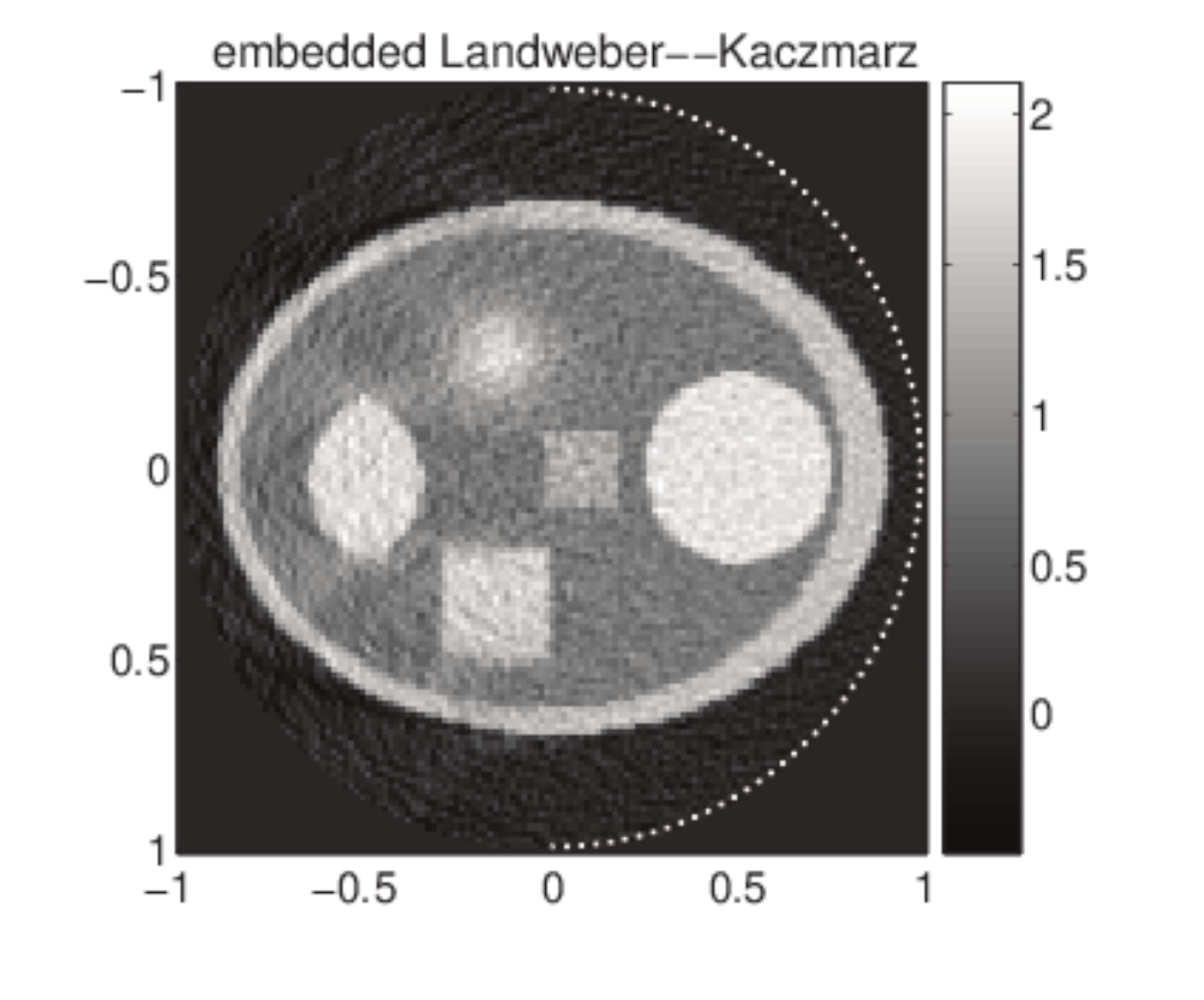}
\includegraphics[height = 0.4 \textwidth, width = 0.5\textwidth]{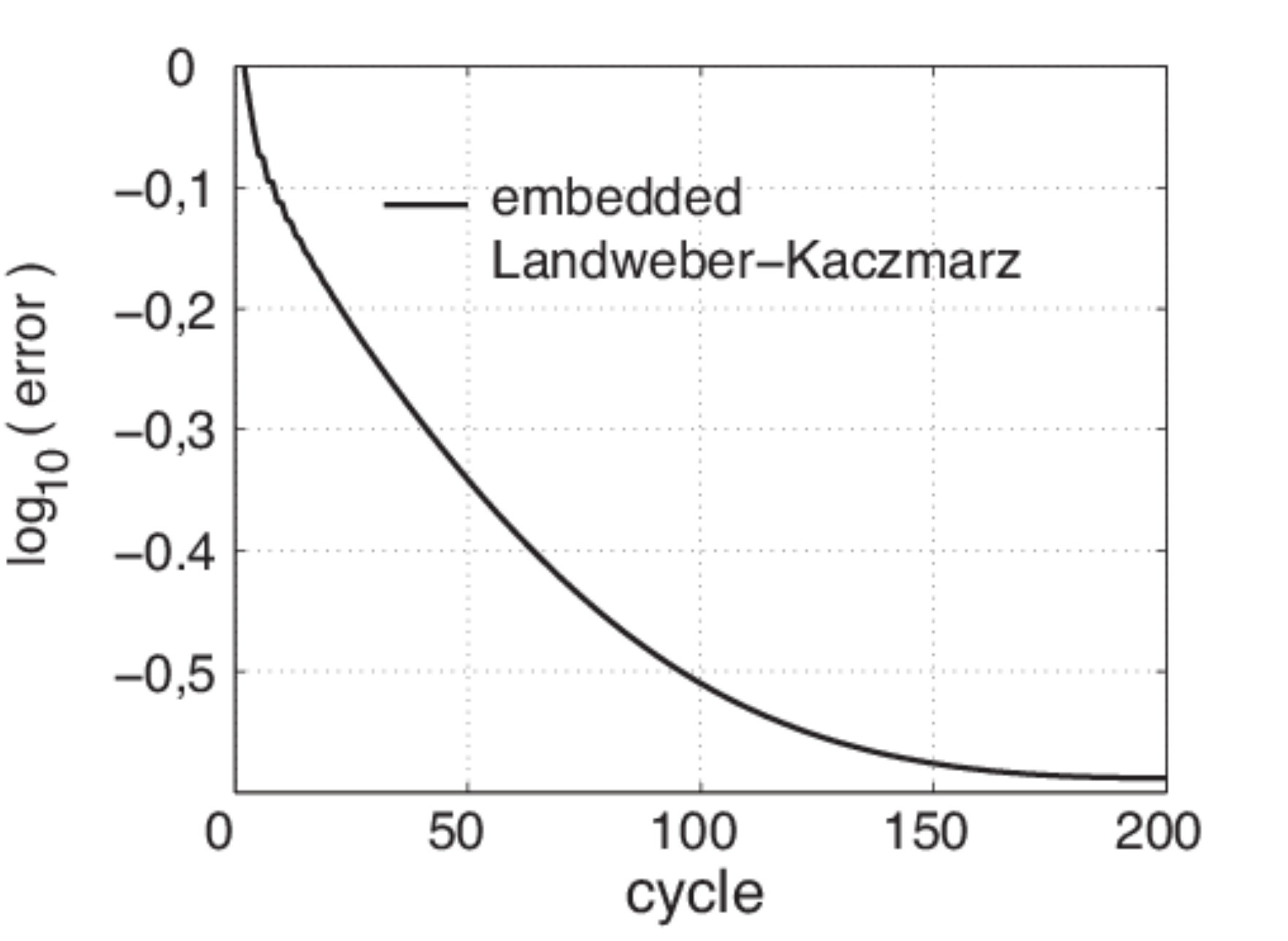}
\caption{Numerical reconstruction of the phantom depicted in
Figure \ref{fg:tat-original} with the \textsc{eLK} method after 230 iterations (left) and
decrease of the error (right).}
\label{fg:tat-embed}
\end{center}
\end{figure}

\subsection{Discussion}

The left image in Figure~\ref{fg:tat-steps} shows the decrease of error for both the
\textsc{LK} method and the \textsc{lLK} method. The number of actually computed Landweber steps in the \textsc{lLK} method
(number of iterations with $\omega_{n}=1$ within one cycle) of the \textsc{lLK} method is shown in the right
image in Figure \ref{fg:tat-steps}. The \textsc{lLK} method provides a better approximation of the exact
solution than the \textsc{LK} method. Moreover, since more than half of the iterates are loped
(see right picture in Figure \ref{fg:tat-steps}), the numerical effort is remarkably smaller.
%
%
Since in the \textsc{eLK} method an averaging over all components of the solution \textit{vector} $\f \x_{n_\star^\delta}$ is performed, the quality of the reconstruction is slightly better than for
the \textsc{lLK} method (and other tested regularization methods).
Shortcoming of the \textsc{eLK} method are the larger amount of memory and the larger number of iteration cycles needed.
%
%
%

\section{An inverse problem for semiconductors}
\label{sec:semicond}

In this section we investigate the solution of an inverse problem for nondestructive
testing of semiconductors. More precisely, we aim to recover the doping profile on
the basis of a simplified drift diffusion model from measurements given by the
voltage--current (VC) map  \cite{LMZ06, LMZ06a}.

The precise implantation of the doping profile is crucial for the desired
performance of semiconductor devices in practice. In order to minimize
manufacturing costs of semiconductors as well as for quality control,
there is substantial interest in replacing expensive laboratory testing by
numerical simulation for non--destructive evaluation.

\subsection{Mathematical modelling}

Let $\Omega \subset \mathbb{R}^d$, $d \in \set{1, 2, 3}$, be a domain representing the
semiconductor device and let $\nu$ denote the unit normal vector to its boundary $\di \Omega$.
The boundary of $\Omega$ is divided into two nonempty  disjoint parts $\partial \Omega_N$
and $\partial \Omega_D$.
The \textit{Dirichlet part} of the boundary $\partial \Omega_D$ models the Ohmic
contacts,  where  an electrostatic potential  is induced.
The \textit{Neumann part} of the boundary $\partial\Omega_N$
corresponds to insulating surfaces.

The \textit{doping profile}  $C: \Omega \to \R$ (unknown parameter)
models the preconcentration of ions in the crystal, which is produced by diffusion of different materials
into the silicon crystal and by implantation with an ion beam.
In particular,  $C = C_{+} - C_{-}$, where $C_{+}$ and $C_{-}$ are
concentrations of positive and negative ions,  respectively.
In those subregions of $\Omega$ in which the preconcentration of negative ions predominate
(P-region), we have $C< 0$. Analogously, we define the N-region,
where $C> 0$. The boundaries between the P-regions and N-regions
(where $C$ changes sign) are called the \textit{PN-junctions},
see Figure~\ref{fig:diode}.

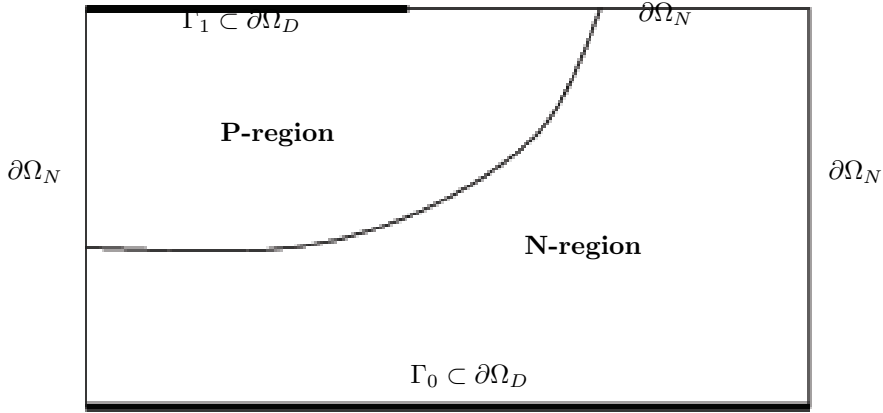
\begin{figure}
\centering
\includegraphics[width=0.8\textwidth]{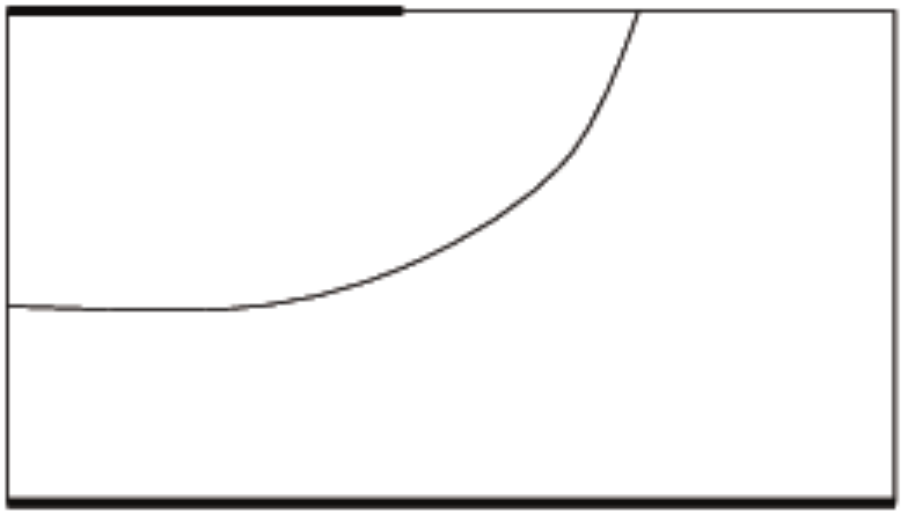}
\caption{P-N diode. Example of P-N junction.} \label{fig:diode}
\vskip-5cm \unitlength1cm
\centerline{
\begin{picture}(13,5)
\put(0.7,4){$\partial\Omega_N$}
\put(11.5,4){$\partial\Omega_N$}
\put(9,6.1){$\partial\Omega_N$}
\put(3,6){$\Gamma_1 \subset \partial\Omega_D$}
\put(6,1.3){$\Gamma_0 \subset \partial\Omega_D$}
\put(3.5,4.5){\bf P-region}
\put(7.5,3.0){\bf N-region}
\end{picture} }
\end{figure}

We review the \textit{linearized stationary unipolar mathematical model}, which is described
by a \emph{drift diffusion equation} (see~\cite{LMZ06, LMZ06a}):
\begin{eqnarray}
    \lambda^2 \, \Delta V  \label{eq:upolV-A}
   =
   \ e^{V} - C(\fs \xi)
   \,, \quad    & \mbox{in}\ \Omega
   \\
   V               \label{eq:upolV-B}
   = \ V_{\rm bi}  \,, \quad         & \mbox{on }  \partial\Omega_D \\
\nabla V \cdot \nu           \label{eq:upolV-C}
   = \ 0               \,, \quad        & \mbox{on }  \partial\Omega_N \, .
\end{eqnarray}
and
\begin{eqnarray} \label{eq:upolU-A}
    \nabla \cdot
    \left[ \mu_n \, e^V \, \nabla u \right]
    =
    0 \,, \quad                      & \mbox{in } \Omega
    \\ \label{eq:upolU-B}
    u
    =
    U(\fs \xi)    \,, \quad          & \mbox{on } \partial\Omega_D
    \\ \label{eq:upolU-C}
    \nabla u \cdot\nu
    = 0
    \,, \quad                        & \mbox{on } \partial\Omega_N
\end{eqnarray}

Here, the function $V:\Omega \to \R$ denotes the electrostatic potential
($- \nabla V$ is the electric field) and $u: \Omega \to \R$ corresponds, up
to an exponential transformation, to the concentration of free carriers of
negative charge (the concentration of free positive charge is assumed to
vanish; the reason why the model is called {\em unipolar}). The positive constant
$\mu_n$ is related to the mobility of electrons
and $\lambda^2$ is the Debye length.

At the  Dirichlet part of the boundary $\partial \Omega_D$
the potential $V$ and the concentration $u$ are prescribed:
$V_{\rm bi}$ is a given logarithmic function \cite{LMZ06a}, and the
function $U$ denotes the applied potential.
At the the Neumann part of the boundary $\partial\Omega_N$
zero current flow (\ref{eq:upolU-C}) and a zero electric field in the
normal direction (\ref{eq:upolV-C}) are prescribed.

In this section we are concerned with determining the {\em inverse doping problem}
for the VC map, which consists in the determination of the doping profile
$C$ in system (\ref{eq:upolV-A}--\ref{eq:upolU-C})
from measurements of the VC map
\begin{eqnarray*}
    \Sigma_C: & H^{3/2}(\partial\Omega_D)
    \to
    \mathbb R
    \, \\
    & U
    \mapsto
    \Sigma_C(U)
    :=
    \mu_n
    \int_{\Gamma_1}
    e^{V_{\rm bi}} \frac{\di u }{\di \nu } \, d\Gamma
\end{eqnarray*}
which maps an  applied potential $U$ at $\partial\Omega_D$ to  the \textit{current flow}
$\Sigma_C(U)$  through the contact $\Gamma_1$.
Here $(V,u)$ solve (\ref{eq:upolV-A}--\ref{eq:upolU-C}) and $\di u / \di \nu$ is the
normal derivative of $u$. For results on existence and uniqueness of
$H^1$-solutions of system (\ref{eq:upolV-A}--\ref{eq:upolU-C}),
we refer to \cite{LMZ06}.

Since systems (\ref{eq:upolV-A}--\ref{eq:upolV-C}) and
(\ref{eq:upolU-A}--\ref{eq:upolU-C}) are decoupled, we can
split the inverse problem in two parts:
First, the inverse problem of identifying the parameter
$\x(\fs \xi) := e^{V(\fs \xi)} $ in (\ref{eq:upolU-A}--\ref{eq:upolU-C})
from measurements of the Dirichlet to Neumann (DN) map
\begin{equation}\label{eq:DN-map}
\begin{aligned}
   \Lambda_\x :
   & H^{3/2}(\partial\Omega_D) \to \mathbb \R
   \\
  U \mapsto &
  \Lambda_\x(U) :=
   \mu_n
   \int_{\Gamma_1}
    \x \, \frac{\partial u}{\partial \nu}
   \, d\Gamma
\end{aligned}
\end{equation}
is solved. The second step consists in the determination of the doping profile
from \req{upolV-A}, namely $C = \x - \lambda^2 \Delta (\ln \x )$.

Notice that the second step is related to twice  numerical
differentiation which is mildly ill-posed (see
\cite{EngHanNeu96}). On the other hand, the first step corresponds
to the inverse problem of impedance tomography (EIT) with partial
data, which is known to be severely ill-posed. (A review of EIT
can be found in~\cite{boc02}.) For the sake of clarity of
presentation we will focus on the problem of identifying $\x$ in
the first step above.

\subsection{Abstract formulation in Hilbert space}

Due to the nature of the practical experiments, some restrictions on
the data have to be taken into account:

\begin{enumerate}
\item
The voltage profiles $U \in H^{3/2}(\partial\Omega_D)$ must
satisfy $U |_{\Gamma_1} = 0$.

\item $\x$ has to be determined from a finite number of
measurements, i.e. from the given data
\begin{equation} \label{eq:data-semicond}
    F_i(\x) := \Lambda_\x(U_i) \in Y:= \mathbb R\,,
    \quad i \in\set{0,\dots, N-1} \,,
\end{equation}
where $U_i \in H^{3/2}(\partial\Omega_D)$
are prescribed voltage profiles satisfying $U_i |_{\Gamma_1} = 0$.
\end{enumerate}

\begin{figure}
\centerline{
\includegraphics[width=0.49\textwidth]{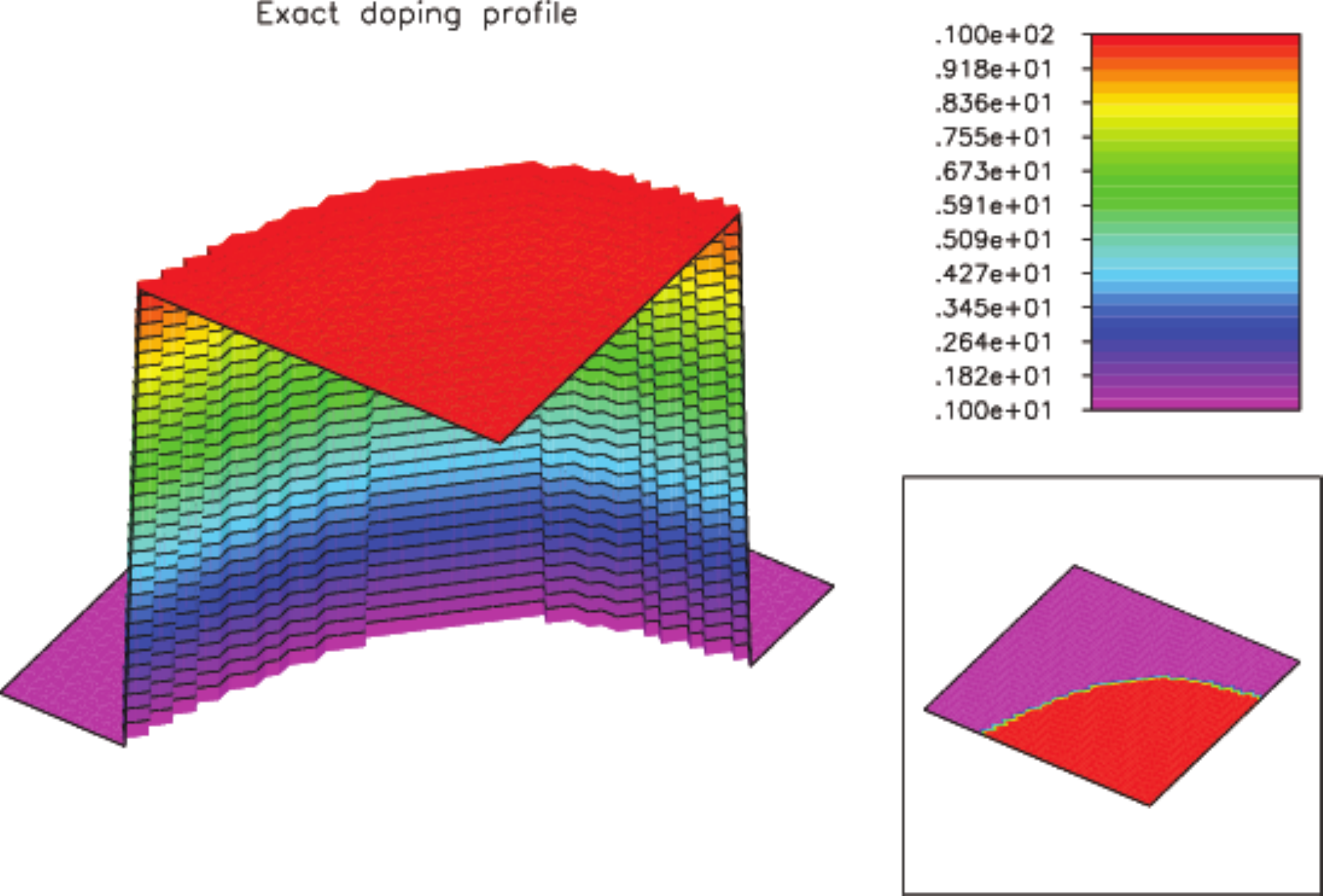} \hskip1cm
\includegraphics[width=0.49\textwidth]{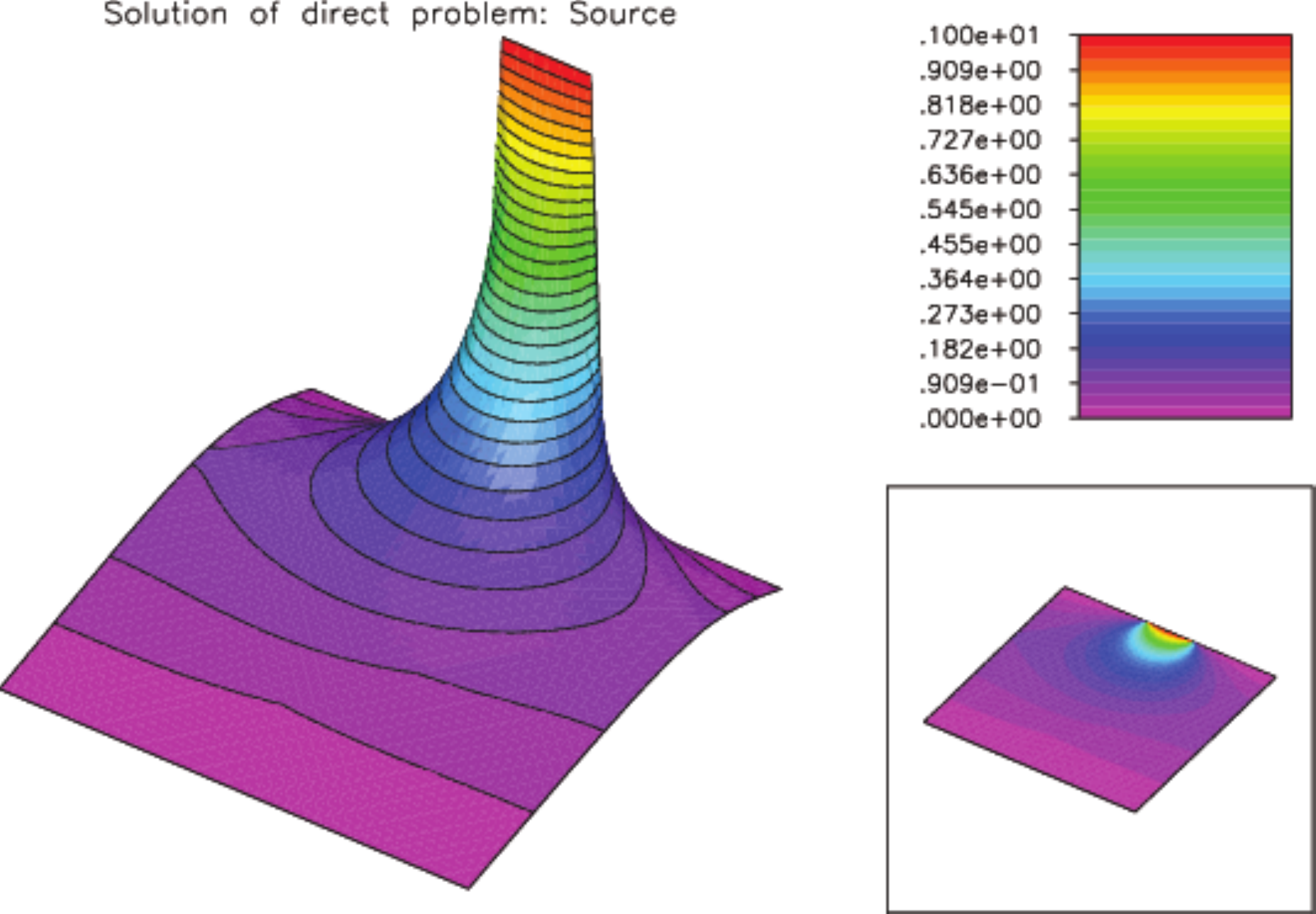} }
\centerline{\hfill (a) \hspace{0.4\textwidth}  (b) \hfill}
\caption{\small Picture (a) show the doping profile to be identified.
Picture (b) shows a typical voltage profile $U_j$ and the corresponding
solution $u$ of (\ref{eq:upolU-A}--\ref{eq:upolU-C}).} \label{fig:idop-setup}
\end{figure}

\begin{figure}
\centerline{
\includegraphics[width=0.49\textwidth]{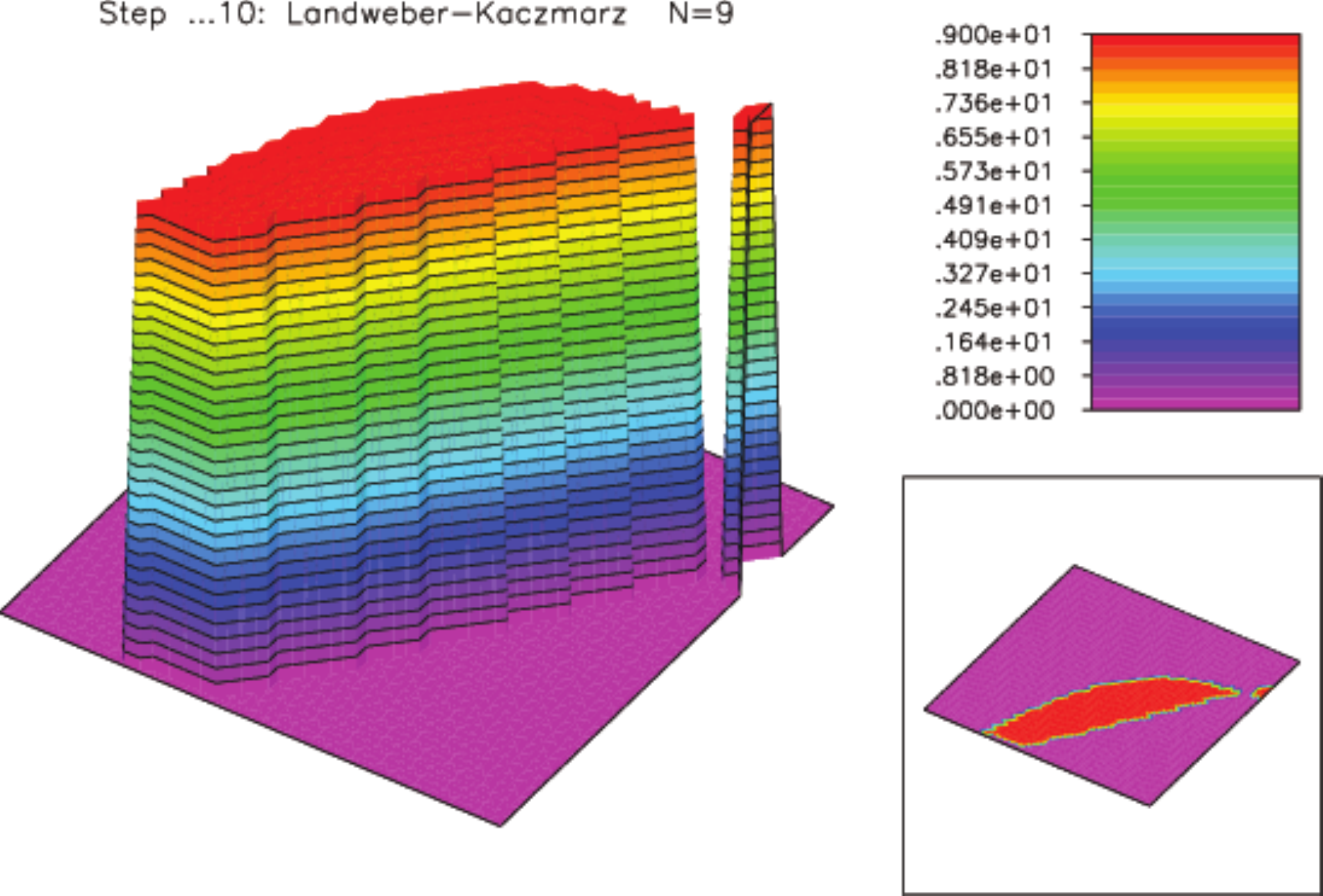}
\includegraphics[width=0.49\textwidth]{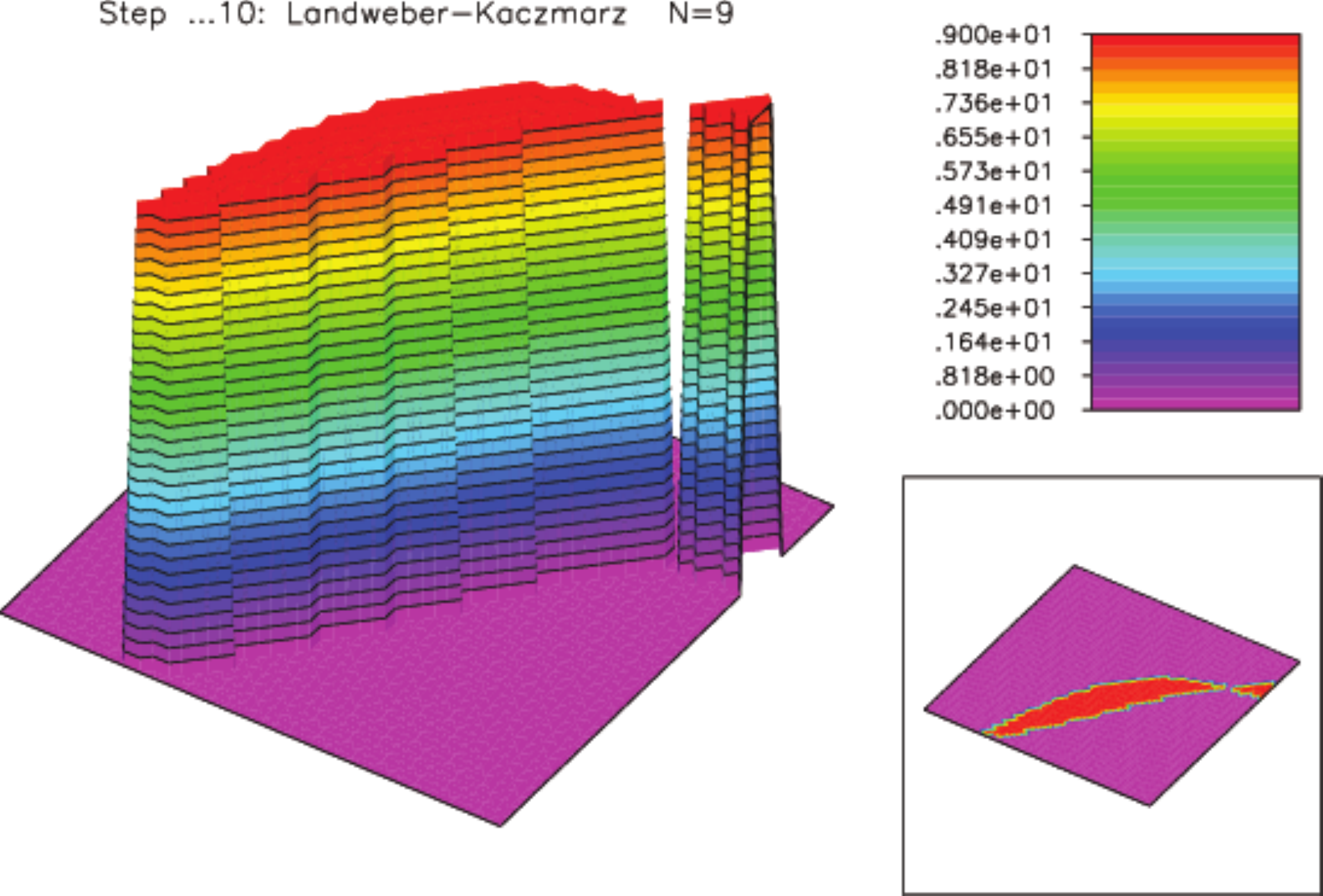}  }
\centerline{
\hfill (a1) \hspace{0.4\textwidth}  (b1) \hfill  \medskip}
\centerline{
\includegraphics[width=0.49\textwidth]{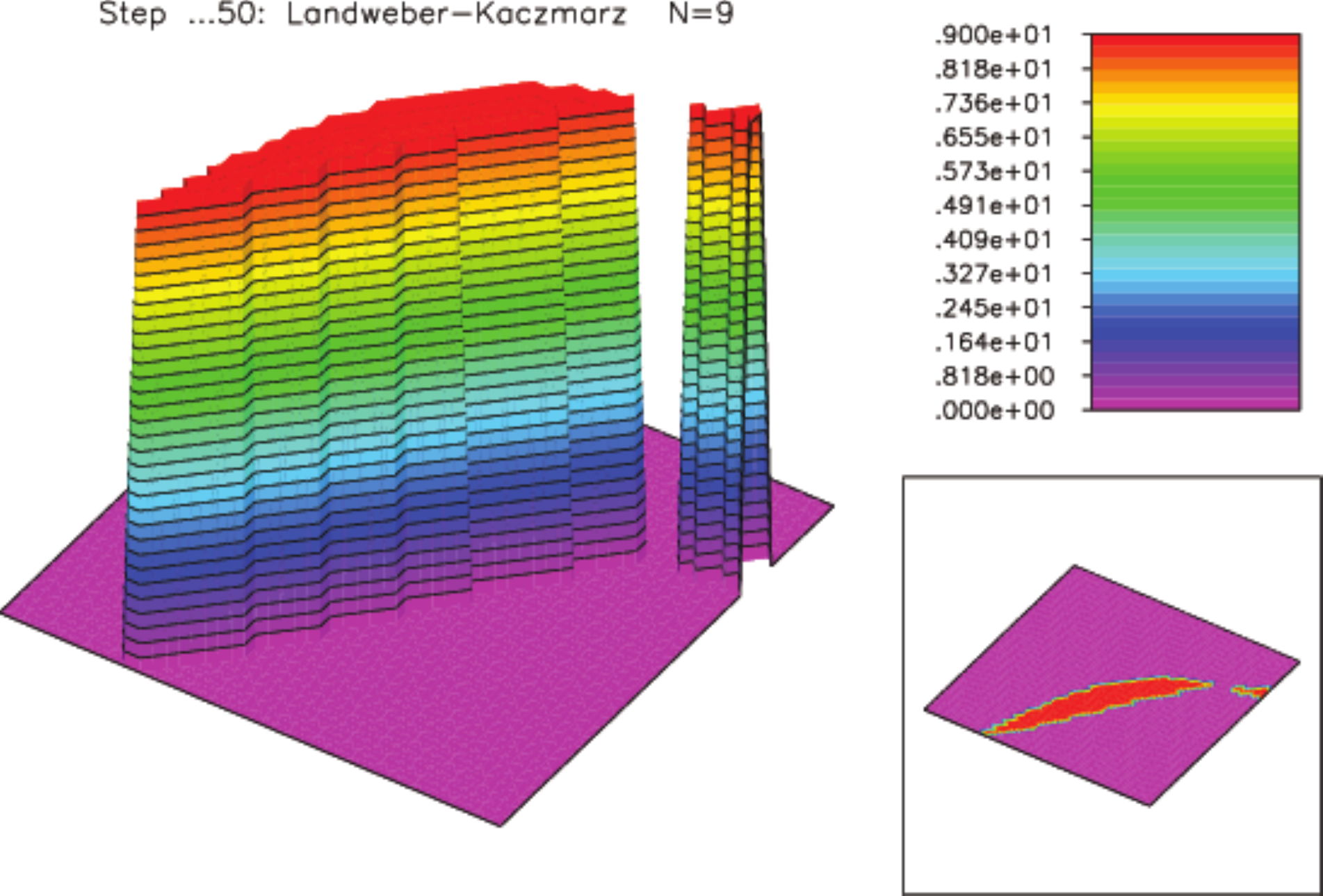}
\includegraphics[width=0.49\textwidth]{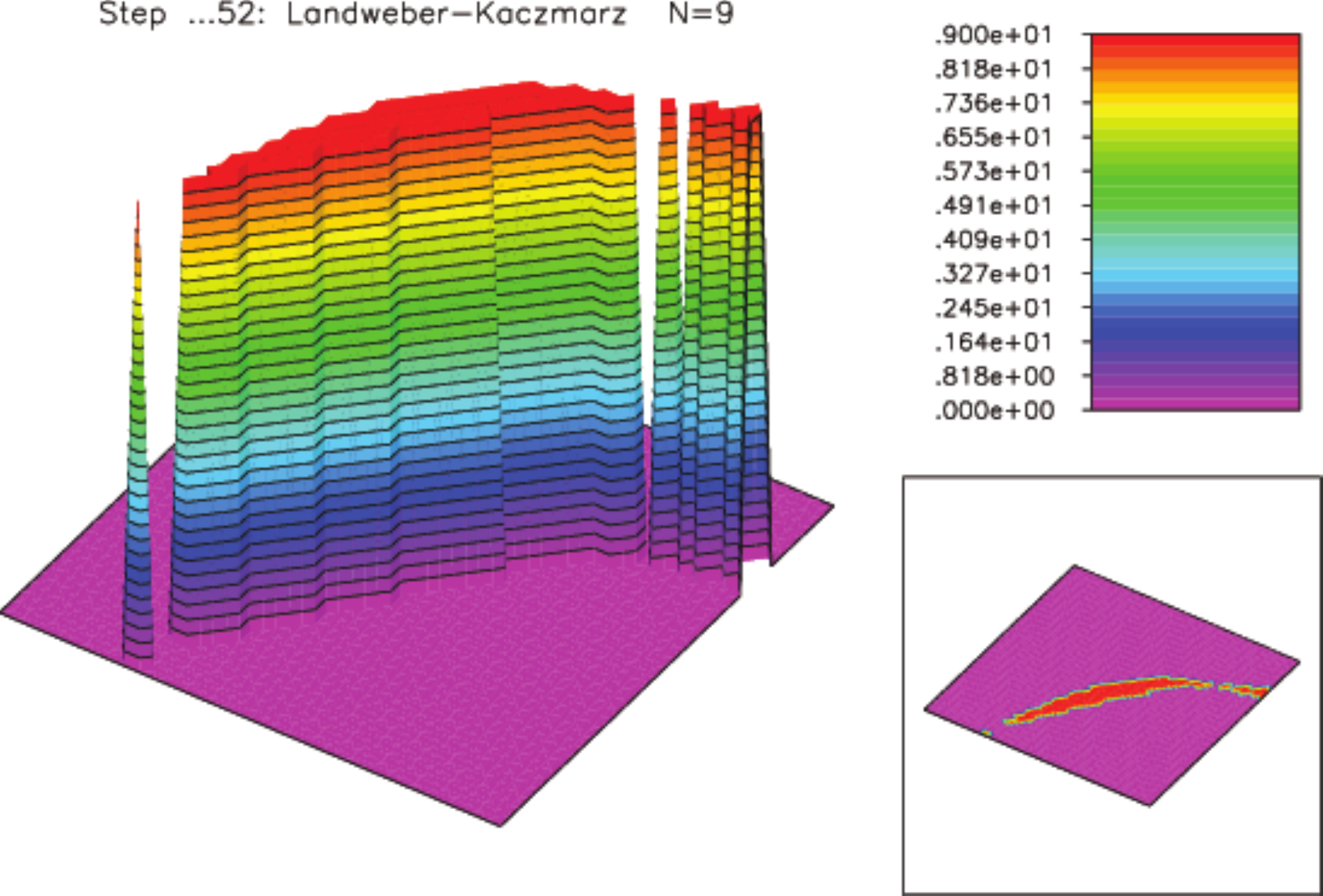} }
\centerline{
\hfill (a2) \hspace{0.4\textwidth}  (b2) \hfill  \medskip}
\centerline{
\includegraphics[width=0.49\textwidth]{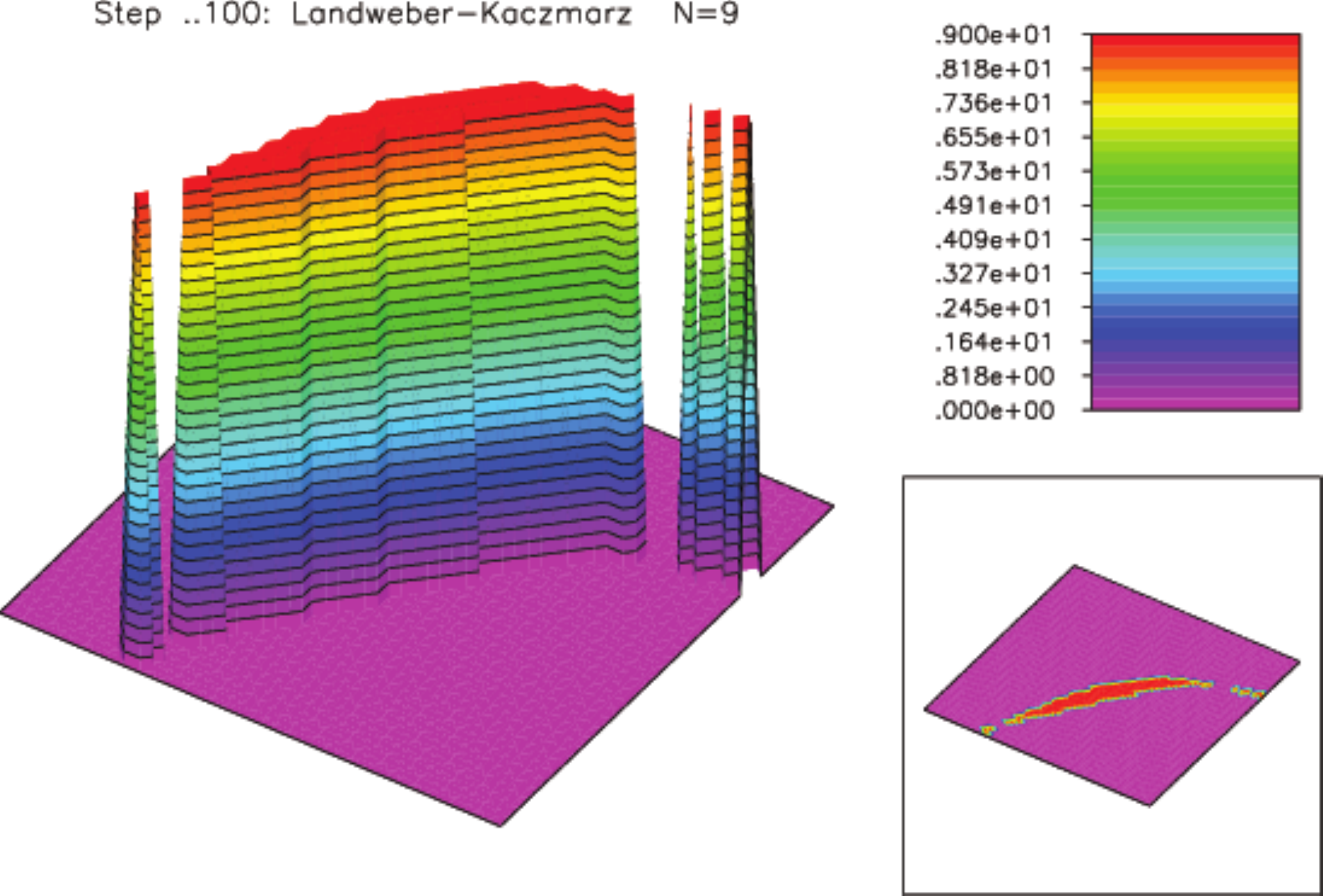}
\hspace{0.05\textwidth}
\includegraphics[width=0.40\textwidth]{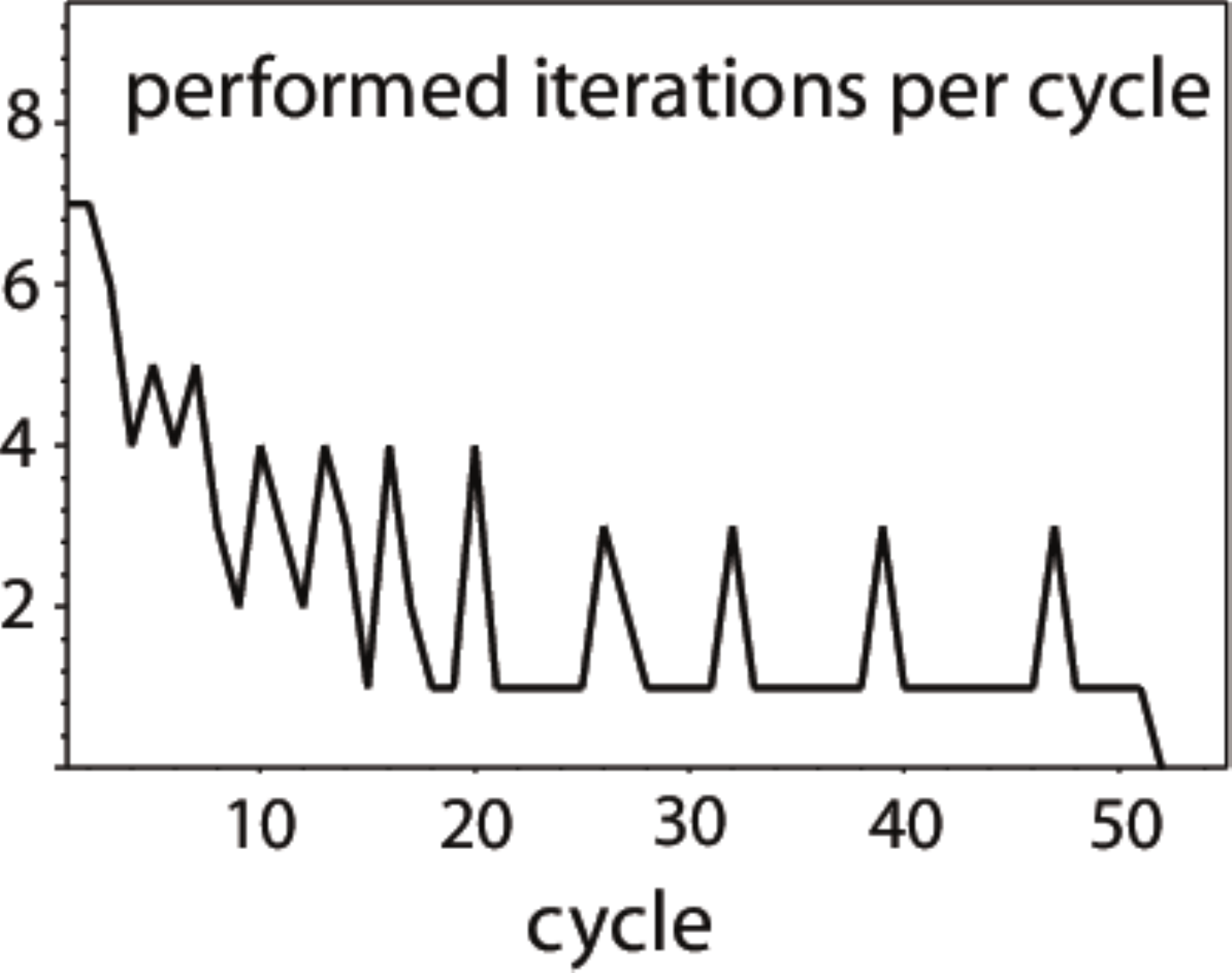}
}
\centerline{ \hfill (a3) \hskip0.45\textwidth  (c) \hfill }
\caption{Numerical reconstruction of the p-n junction in
Figure~\ref{fig:idop-setup}~(a).
Pictures (a1--a3) show the decrease of the iteration error for the \textsc{LK} method. Pictures
(b1, b2) show the corresponding evolution for the \textsc{lLK} method, and picture (c) the
number of computed Landweber steps in each cycle of the \textsc{lLK} method.}
\label{fig:idop-iter}
\end{figure}

In order to apply the abstract framework of \cite{HLS06}  we write the inverse 
doping problem as a system of  operator equations
\[
     \F_i(\x) = y^{\delta, i} \,, \quad i = 0,\dots, N-1 \,.
\]
Here $\x \in L^2(\Omega) =: X$ is the unknown parameter, $y^{\delta, i} \in \R$
denotes measurement data and $\F_i: X \to Y$ are the parameter to output maps,
with domain of definition
\[
    D(\F_i) := \{ \x \in L^\infty(\Omega) :  \, 0 < \x_{\rm min}
   \le \x  \le \x_{\rm max}, \mbox{ a.e.} \} \, .
\]

Although the operators $\F_i$ are Fr\'echet differentiable, they do not satisfy
the tangential cone condition \cite[Eq. (13)]{HLS06}. Therefore, the
convergence results derived in \cite[Section~2]{HLS06} cannot be
applied.

\subsection{Numerical reconstruction}

In the following numerical examples we assume that $N = 9$
Dirichlet--Neumann pairs $(U_i, F_i(\x) )$ of measurement data are available.
The domain $\Omega \subset \mathbb R^2$ is the unit
square, and the boundary parts are defined as follows
\begin{equation}
\begin{aligned}
    \Gamma_1  &:=   \set{ (s,1) : \ s \in (0,1)    } \, ,
    \\
    \Gamma_0  &:=   \set{ (s,0) :   \ s \in (0,1)  } \, ,
    \\
    \partial\Omega_N
    &:=
    \set{ (0,t) : \ t \in (0,1) }
    \cup
        \set{ (1,t) : \ t \in (0,1) } \, .
\end{aligned}
\end{equation}
The fixed inputs $U_j$, are chosen to be piecewise constant functions
supported in $\Gamma_0$
\[
  U_j(s) \ := \ \left\{ \begin{array}{rl}
      1, & |s - s_j| \le h \\
      0, & {\rm else} \end{array} \right.
\]
where the points $s_j$ are uniformly distributed on $\Gamma_0$ and
$h = 1/32$. The doping profile to be reconstructed is shown in Figure~
\ref{fig:idop-setup}~(a). In Figure~\ref{fig:idop-setup}~(b) a typical
voltage profile $U_j$ (applied at $\Gamma_0$) is shown as well as the
corresponding solution $u$ of (\ref{eq:upolU-A}--\ref{eq:upolU-C}).
In these pictures, as well as in the forthcoming ones, $\Gamma_1$ appears
on the lower left edge and $\Gamma_0$ on the top right edge (the origin
corresponds to the upper right corner).

In Figure~\ref{fig:idop-iter} we show the evolution of the iteration error for both
the \textsc{LK} method (a1--a3) and the \textsc{lLK} method (b1, b2). The number of actually computed
Landweber steps within each cycle of the \textsc{lLK} method is shown in
Figure~\ref{fig:idop-iter}~(c).

The stopping rule for the \textsc{lLK} method with
$\tau = 2.0$ is satisfied after 52 cycles.
In Figures~\ref{fig:idop-iter}~(b1, b2) one can see the iteration error
after 10 and 52 cycles. For comparison purposes, the iteration error for
the \textsc{LK} method is plotted in Figures~%
\ref{fig:idop-iter}~(a1--a3) after 10, 50 and 100 cycles.

\subsection{Discussion}

The \textsc{LK} method requires almost twice as much cycles as the
\textsc{lLK} method in order to obtain a similar accuracy. The
efficiency of the \textsc{lLK} method becomes even more evident
when we compare the total number of actually performed iterations.
Each cycle of the \textsc{LK} method requires the computation of 9
iterations, while in the \textsc{lLK} method the number of
actually performed iterations is very small after a few iteration
cycles. As one can see in Figure \ref{fig:idop-iter} (c),  no more
than 3 Landweber steps are computed after the $20$--th cycle of
the \textsc{lLK} method. In total for the computation of the
approximation in Figure \ref{fig:idop-iter} (a3), 900 iterations
are needed, while the approximation in (b2) requires the
computation of only 146 iterations.

\section{Reconstruction of transducer pressure fields from Schlieren images}
\label{sec:sch}

This section is concerned with the problem of reconstructing three dimensional
\textit{pressure fields} generated by a medical ultrasound transducer
from \textit{Schlieren data}.

The data are collected with a Schlieren optical system where in
the experiment an acoustic pressure is emitted into a  water tank (see Figure \ref{OpticalSystem}). The
Schlieren optical system outputs the intensity pattern of light
passing through the tank which is proportional to the square of
the line integral of the pressure along the light path.

For practical aspects on the realization of Schlieren systems and more background information on
such measurement devices we refer
to \cite{Bre98, ChaPel95, HanZan94, LedZan99, PitGreLuKin94, RamNat35, ZanKad94}.

\begin{figure}
\begin{center}
\includegraphics[width=0.7\textwidth]{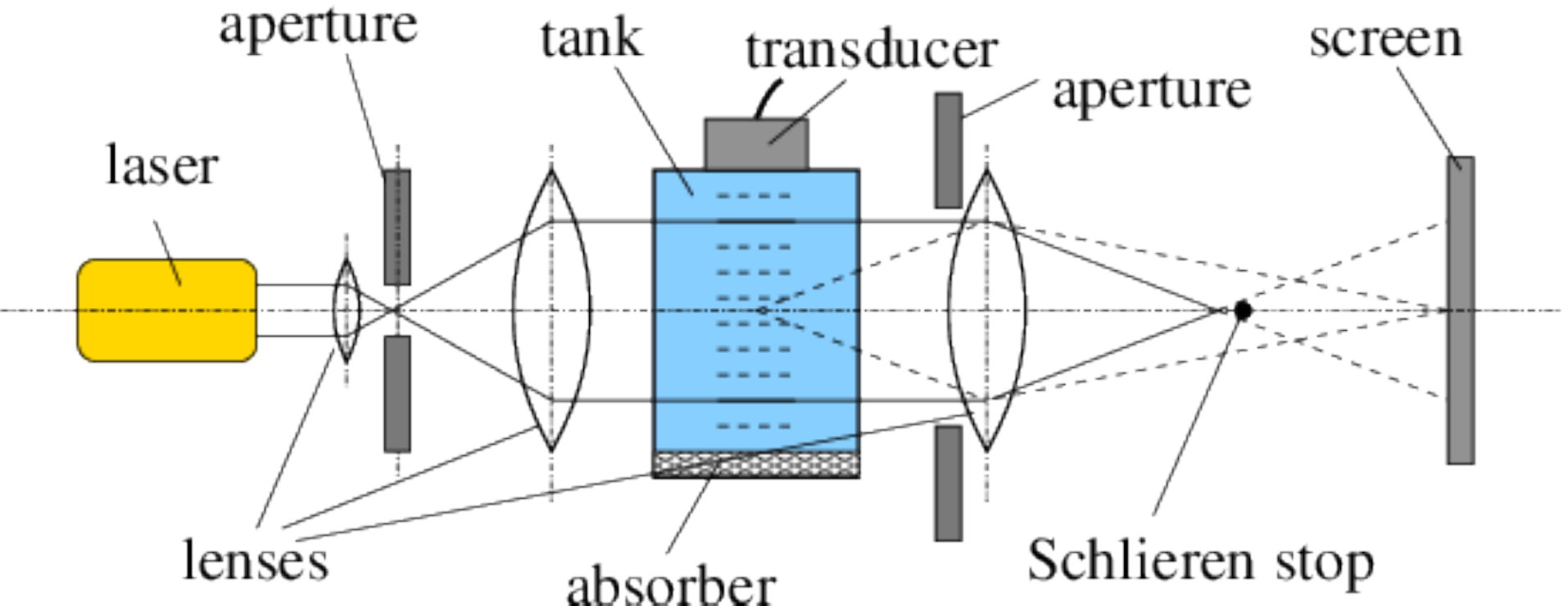}
\end{center}
\caption{Schlieren optical system. The pressure field
in the tank is reconstructed from Schlieren data
acquired at different angles $\sigma_i$, $i \in \set{0, \dots, N-1}$
of rotation of the schlieren system.}
\label{OpticalSystem}
\end{figure}

\subsection{Mathematical modelling}

Let $H$ and $D:=\{\fs \xi\in\R^2\,: \abs{\fs \xi} < 1\}$ denote the \emph{height} and an
arbitrary \emph{cross-section} of the tank, respectively.
The pressure (to be reconstructed) within the tank is represented by a function
$p:  D \times [ 0, H ] \to \R$.
At each \textit{recording angle}  $\fs \sigma_i \in S^1$, $i = 0,\dots, N-1$
the Schlieren system outputs
\begin{equation*}
    P_i ( s, z ) :=
        \left(
            \int_{\R}
            p( s \fs \sigma_i + r \fs \sigma_i^\bot, z )
        \, dr
        \right)^2\;, \quad  z \in [0,H] \,, \ s \in I\,,
\end{equation*}
where $I := [-1,1]$ and the parameter $s$ corresponds to the signed distance of
the line $L_i(s) := s \fs \sigma_i + \R \fs \sigma_i^\bot$ from the
origin in $D$,  see Figure \ref{fig_domain}.
In the following the height parameter $z$ is fixed.  We aim to reconstruct the
\textit{pressure function} $\x: D \to \R$  defined by  $\x(\fs \xi) := p(\fs \xi, z) $
from data $F_i(\x)(s)  := P_i(s,z)$.

%
\begin{figure}
\begin{center}
\includegraphics[height=5.5cm,angle=0]{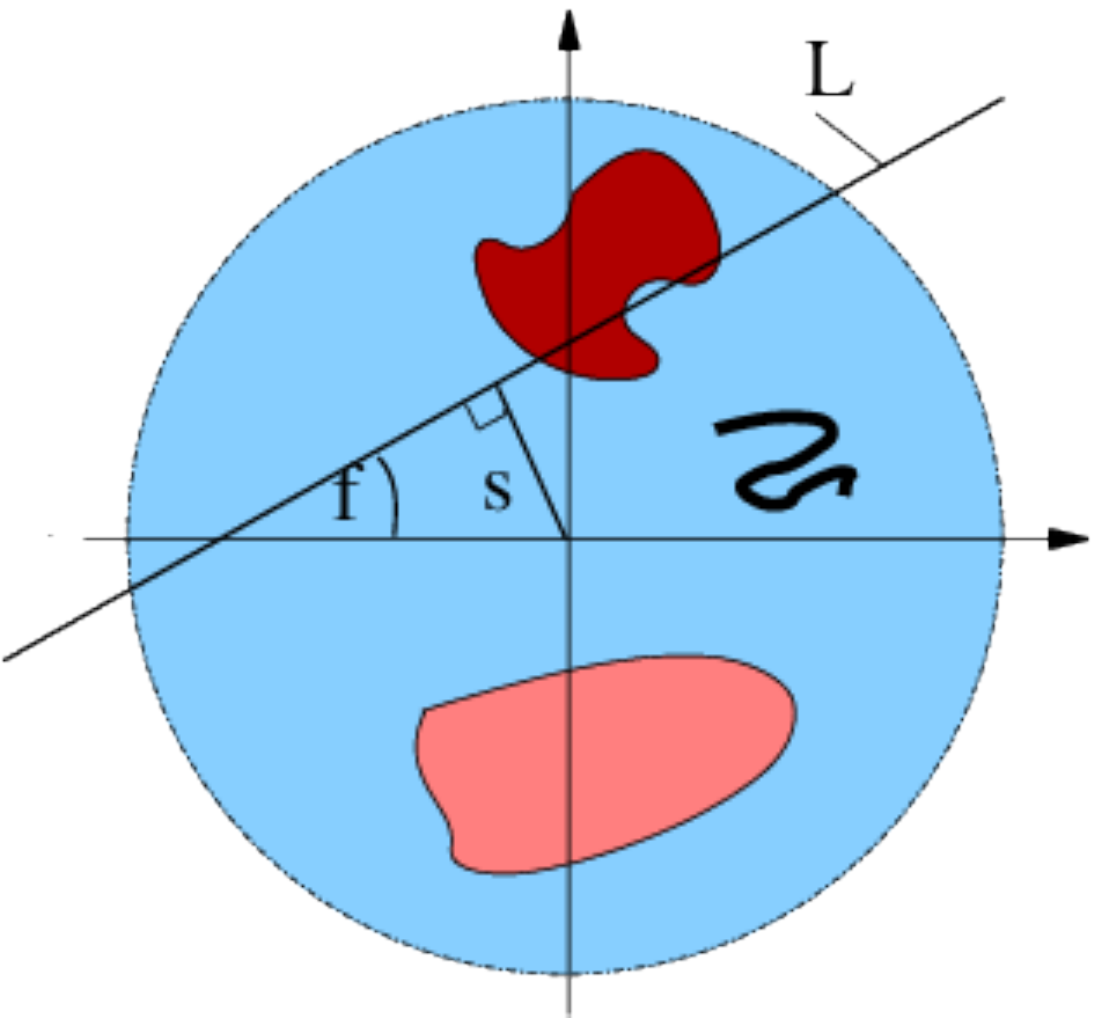}
\end{center}
\caption{Cross section of the tank.
The Schlieren system outputs squares of integrals over the lines $L_i(s)$,
with unit normal $\fs \sigma_i = (\cos \varphi, \sin \varphi)$ and signed
distance $s$ from the origin.}
\label{fig_domain}
\end{figure}
%

Now \textit{Schlieren tomography} can be formulated as the problem of solving the system of
operator equations
\[
    F_i(\x) = \y^{\delta, i} \,, \quad i = 0,\dots, N-1 \,,
\]
for given  noisy data $\y^{\delta,i}$.

\subsection{Abstract formulation in Hilbert--space}

Let
\begin{equation*}
    \Ro_i: C_0^\infty (D) \to C_0^\infty ( I ): \quad
    \x  \mapsto
    \Ro_i(\x) :=
    \left(
        s  \mapsto
                \int_{\R}
                    \x( s \fs \sigma_i + r \fs \sigma_i^\bot )
                \, dr
    \right)
\end{equation*}
denotes the \textit{Radon transform}, and $F_i = \Ro_i^2$ denotes the \emph{Schlieren transform}.

\begin{theorem}
The operators $\F_i$, $ i \in \set{0, \dots, N-1}$ can be continuously extended on
$H_0^1(D)$:
\[
    \F_i: H_0^1(D)  \to     L^2(I) \;.
\]
Moreover, $\F_i$ is Fr\'echet differentiable and the Fr\'echet derivative
of $\F_i$ at $\x$ is given by
\begin{equation}\label{eq:sch-derivative}
\begin{aligned}
    \F_i^\prime[\x](h)
    &=
    2 \Ro_i(\x) \Ro_i(h)\,.
\end{aligned}
\end{equation}
\end{theorem}

\begin{proof}
Let $\x\in C_0^\infty (D)$.
By applying the Cauchy Schwarz inequality  two times,
\begin{equation}\label{estH}
\begin{aligned}
 \norm{\F_i( \x )}^2_{L^2}
  & =
  \int_{-1}^1
  \left[
  \left(
    \int_{\R}
    \x( s \fs \sigma_i + r \fs \sigma_i^\bot )
    \chi_D( s \fs \sigma_i + r \fs \sigma_i^\bot  )
    \, dr
  \right)^2
  \right]^2 ds
  \\
 & \leq \int_{-1}^1
  \left(
    2 \int_{\R}
    \x( s \fs \sigma_i + r \fs \sigma_i^\bot )^2
    \chi_D( s \fs \sigma_i + r \fs \sigma_i^\bot  )
    \, dr
  \right)^2 ds
 \\
  & \leq
  8\int_{-1}^1
  \left(
    \int_{\R}
    \x( s \fs \sigma_i + r \fs \sigma_i^\bot )^4
    \, dr
  \right) ds
  = 8 \norm{ \x }^4_{L^4} \,.
\end{aligned}
\end{equation}
Using  Sobolev's embedding theorem (see~\cite{Ada75}), it follows that
\begin{equation}\label{eq:L4H1}
    \norm{\F_i( \x )}^2_{L^2} \leq 8 \cdot \norm{\x}_{L^4}^4 \leq C \,\norm{\x}_{H^1}^4
\end{equation}
for some positive constant $C$. Therefore, $\F_i$ extends, by continuity,
to an operator $H_0^1(D) \to L^2(D)$.

Now let $\F_i'[\x]$ be defined by \req{sch-derivative}.
Obviously $\F_i'[\x]$ is a bounded linear operator. Moreover,
from the Cauchy Schwarz inequality it follows that
\begin{equation*}
\begin{aligned}
 \norm{ \F_i(\x + h) - \F_i(\x) - \F_i'[\x](h) } _{L^2}^2
 & =
  \norm{ \Ro_i^2(\x + h) - \Ro_i^2(\x) - 2 \Ro_i(\x) \Ro_i(h) } _{L^2}^2
  \\
  &
  =
  \norm{ R^2_i(h)}_{L^2}^2
  =
  \norm{ \F_i(h)}_{L^2}^2
  \leq  C \|h\|_{H^1}^4 \,.
\end{aligned}
\end{equation*}
Therefore, we have
\[
    \lim _{h\to 0} \norm{ \F_i(\x + h) - \F_i(\x) - \F_i'[\x](h) }_{L^2} / \norm{h}_{H^1} = 0 \,,
\]
which implies that  $\F_i'[\x]$ is the Fr\'echet
derivative of $\F_i$ at $\x$.
\end{proof}

We denote by $\Ro_i^\sharp$ the adjoint of $\Ro_i$, considered as an operator from $L^2( D )$ into
$L^2(I)$ \cite{Nat01}, which is given by
\begin{equation}\label{eq:bp}
    \Ro_i^\sharp: L^2(I) \to L^2(D) \quad
    \y \mapsto
    \bigg(
        \fs \xi \mapsto
        \Ro_i^\sharp(\y)(\fs \xi)
        :=
        \y(\langle \fs \xi, \fs \sigma_i \rangle)
    \bigg) \,.
\end{equation}
With this operator we can define the adjoint of $F_i'[x]$:

\begin{theorem}\label{th:adjoint}
The adjoint  $\F_i^\prime[ \x ]^\ast: L^2(I) \to H_0^1(D) $ of
$\F_i^\prime[\x]$ is given by $\y \mapsto \F_i^\prime[ \x ]^\ast(\y) =: q$, where $q$ is the
unique solution of
\begin{equation}\label{eq:adjoind}
   \left(  \mathrm{I} - \Delta \right) q  = 2 \Ro_i^\sharp ( \Ro_i(\x) \y )
\end{equation}
in $H_0^1(D)$. Here $\Delta$ denotes the \textit{Laplace operator}
on $H_0^1(D)$.
\end{theorem}

\begin{proof}
Let $\x\in H^1_0(D)$. We remark that $\F_i^\prime[\x]^\ast$ is
defined by
\begin{equation}\label{eq:adjoint1}
  \langle \F_i'[\x](h), \y \rangle_{L^2}
  =
  \langle h, \F_i'[\x]^\ast(\y) \rangle_{H^1}
\end{equation}
for all $y\in L^2(D)$ and $\x\in H^1_0(D)$.

\begin{enumerate}
\item
Using Fubini's theorem and \req{bp}, the left hand side of \req{adjoint1} can be
rewritten as follows
\begin{equation}\label{adjoint2}
\begin{aligned}
  \langle \F_i'[\x](h), \y \rangle_{L^2}
   & = 2 \int_{-1}^{1} \left( \y(s) \Ro_i(\x)(s) \int_{\R}
                    h( s \fs \sigma_i + r \fs \sigma_i^\bot )
                \, dr \right) ds
   \\
   & = 2 \int_B
    h(\fs \xi) \y(\bracket{\fs \xi, \fs \sigma} ) \Ro_i(\x)(\bracket{\fs \xi, \fs \sigma})
   \dd \fs \xi \\
   & =
   \bracket{ h, 2 \Ro_i^\sharp( \Ro_i(\x) \y)} _{L^2}\,.
\end{aligned}
\end{equation}

\item
The left hand side of \req{adjoint1} reads as follows
\begin{equation}\label{adjoint3}
\begin{aligned}
  \bracket{ h, \F_i'[\x]^\ast(\y) }_{H^1}
    & =
   \bracket{ h, \F_i'[\x]^\ast(\y)}_{L^2}
   +
   \bracket{ \nabla h, \nabla \F_i'[\x]^\ast(\y) }_{L^2}
   \\
  &=
  \bracket{ h, \F_i'[\x]^\ast(\y) - \Delta \F_i'[\x]^\ast(\y) }_{L^2}
  \\
  &=
  \bracket{ h, (1 - \Delta ) \F_i'[\x]^\ast(\y) }_{L^2}
\end{aligned}
\end{equation}
since $h$ vanishes at the boundary of $D$.
\end{enumerate}
Inserting   (\ref{adjoint2}) and (\ref{adjoint3}) in \req{adjoint1}
concludes the proof.
\end{proof}

\subsection{Numerical reconstruction}

According to \req{adjoind} the \textsc{lLK} method for
Schlieren tomography reads as follows
\begin{equation*}
  \x_{n+1} =
    \x_{n} -
    2 \mu^2 \omega_n
    (\mathrm{I}-\Delta)^{-1} R^\sharp_{[n]}
    \bigg(
        \Ro_{[n]}(\x_{n}) ( \F_{[n]}(\x_{n}) - \y^{\delta,[n]})
    \bigg)\,,
\end{equation*}
and
\begin{equation} \label{eq:skipR}
\om_n := \om_n( \delta^{[n]}, \y^\delta)   =
\begin{cases}
      1  & \norm{ \F_{[n]}(\x)  - \y^{\delta, [n]}}
      > \tau \delta^{[n]}, \\
      0  & \mbox{ otherwise\,. }
\end{cases}\,.
\end{equation}
Here $2 (\mathrm{I}-\Delta)^{-1}  \Ro_{[n]}^\sharp ( \Ro_{[n]}(\x) \y )$ denotes the
unique solution of \req{adjoind} with $\y = ( \F_{[n]}(\x_{n}) - \y^{\delta,[n]})$ and $\mu$
is a scaling parameter that ensures $\norm{\mu \F_i'[\x]} \leq 1$ in a closed ball
around the starting value $\x_0$.

%
\begin{figure}
\begin{center}
\includegraphics[height=4.9cm,angle=0]{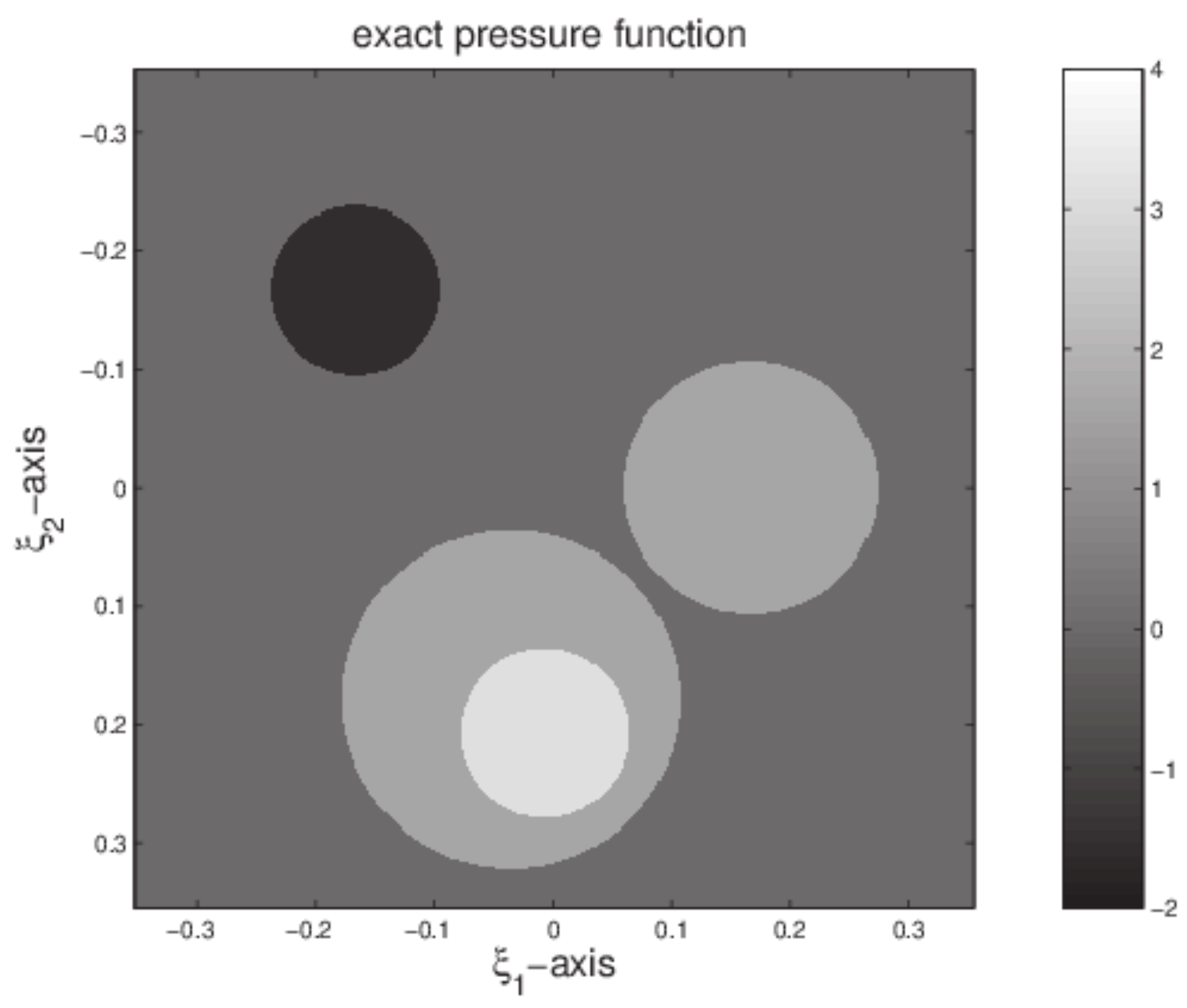}
\includegraphics[height=4.9cm,angle=0]{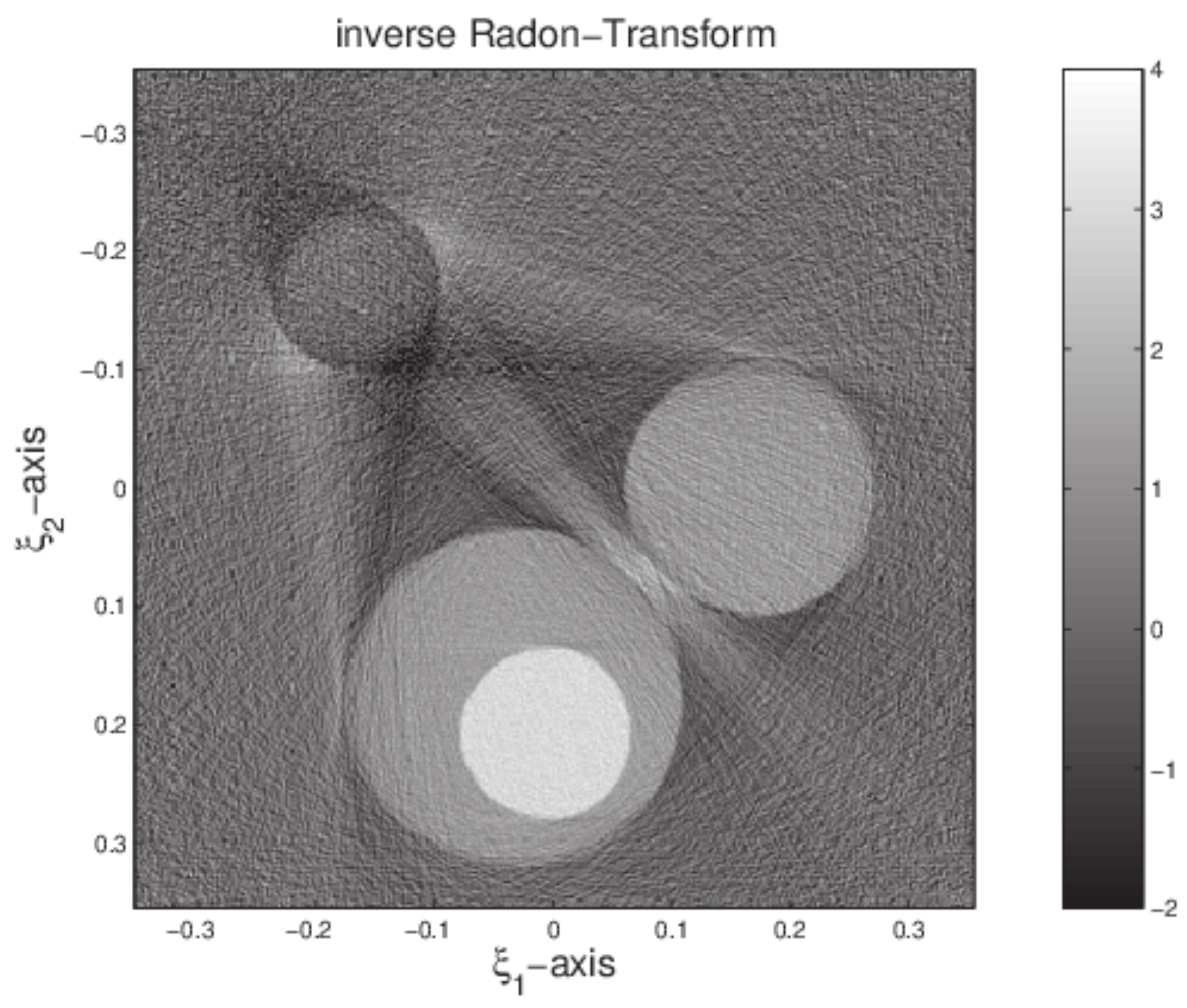}
\includegraphics[height=4.9cm,angle=0]{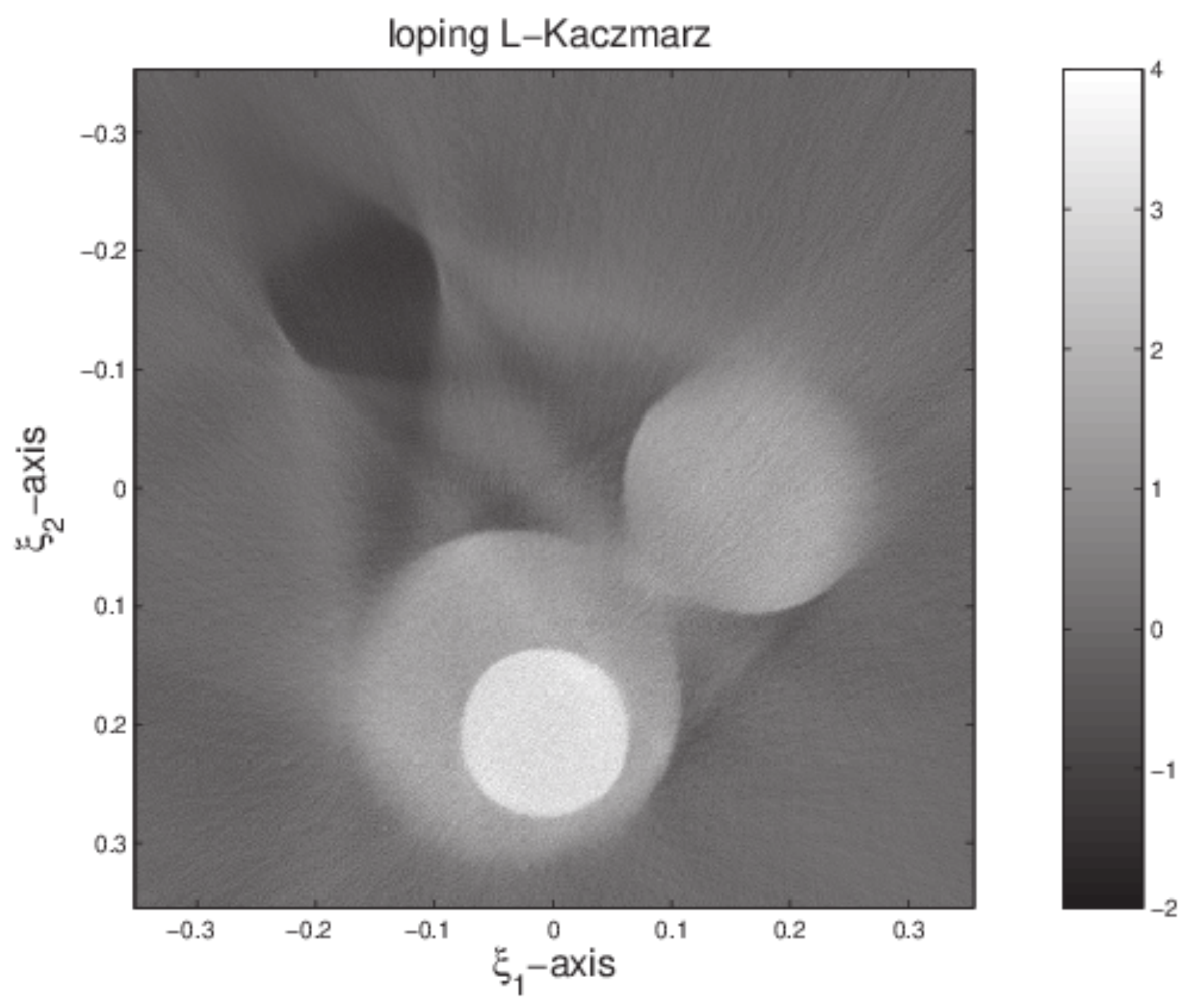}
\includegraphics[height=4.9cm,angle=0]{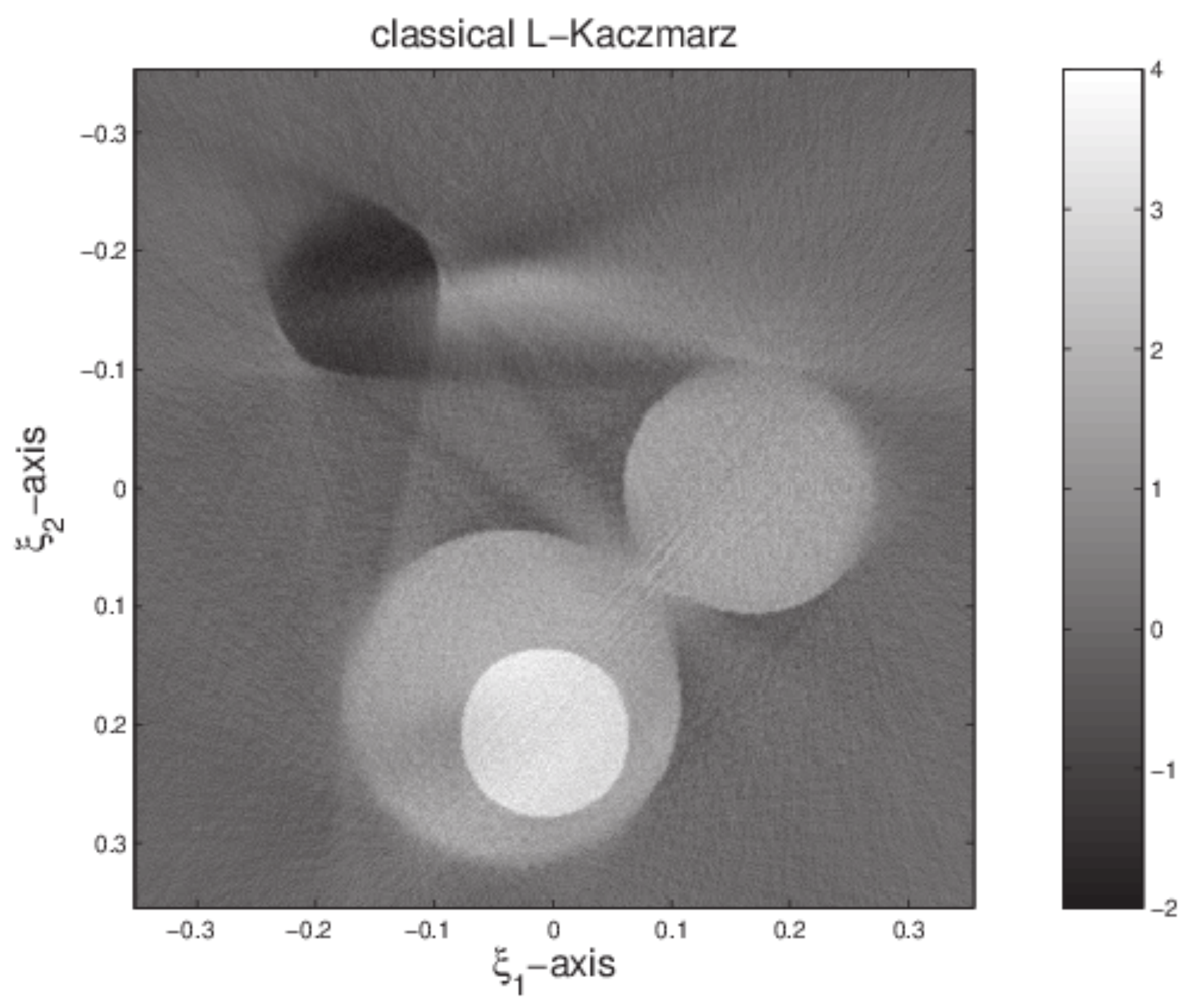}
\end{center}
\caption{The pictures in the first line show the exact pressure function $\x$
and the estimation with the Filter Back Projection (FBP) Algorithm from noisy 
data.
The pictures in the second line show the results obtained with the
\textsc{lLK} method and the \textsc{LK} method with $\tau=2.2$.
All three simulations were performed with data with was perturbed by $0.01$ 
percent noise.
Note that in our reconstruction negative values in $\x$ can be detected, which cannot be 
detected by applying a FBP algorithm.}
\label{sch-rekon}
\end{figure}

In the numerical implementation we approximated functions defined on
$D$  and $I$ by piecewise linear splines. In the numerical experiment we used
$\tau = 2.2$ and $N = 250$.
The synthetic data set $\y^{\delta, i}$, $i \in 0, \dots, N-1$,
were generated by adding normal distributed random noise with $0.01 \%$ noise level
to the exact data.
The phantom to be reconstructed is shown in Figure \ref{sch-rekon}. It is a
superposition of several  characteristic functions. The reconstructions with the
\textsc{lLK} method and the \textsc{LK} method were performed using the constant function
$\x_0(\fs \xi) = 0.01$ as initial guess.

\begin{figure}
\begin{center}
\includegraphics[height=5.0cm,angle=0]{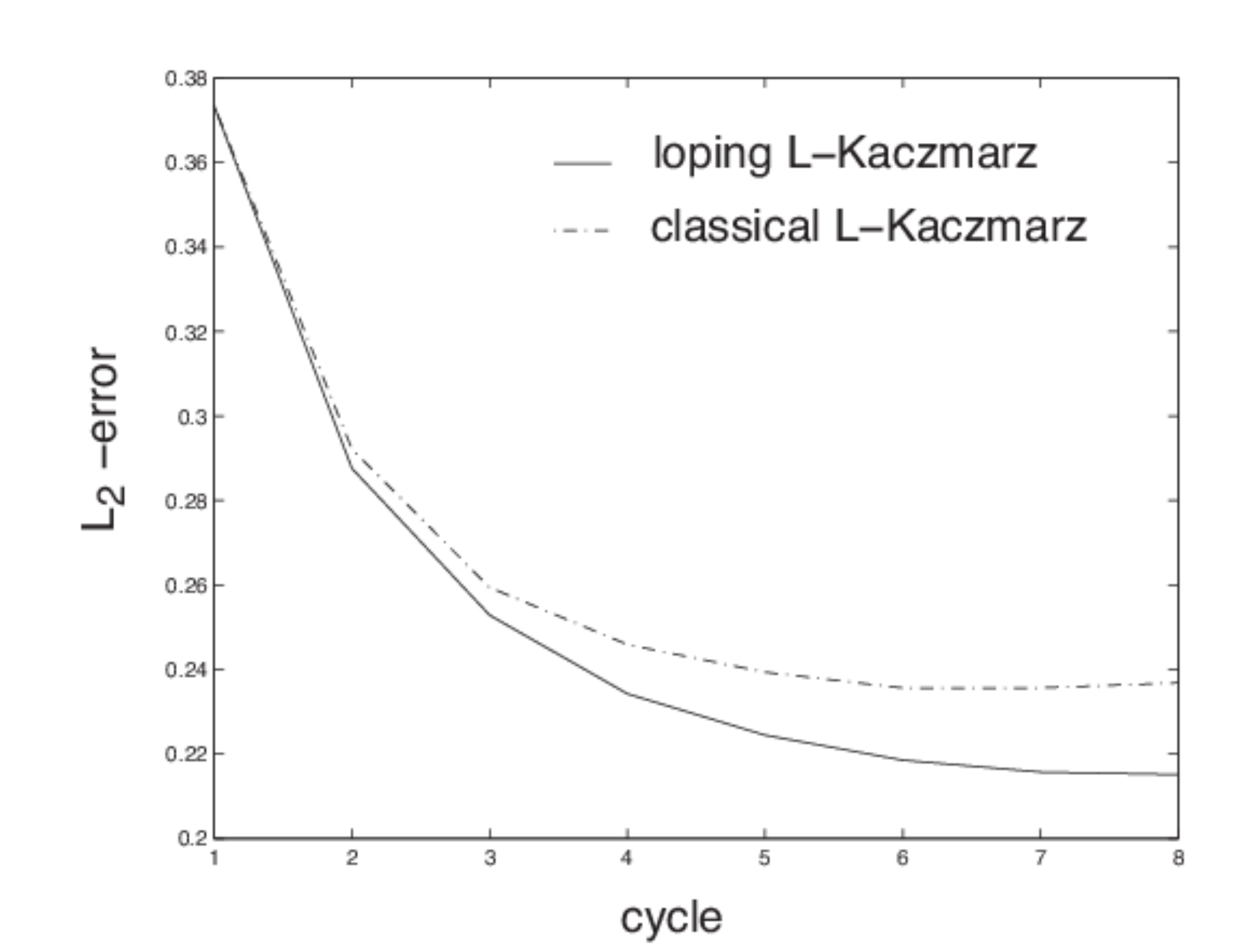}
\includegraphics[height=5.0cm,angle=0]{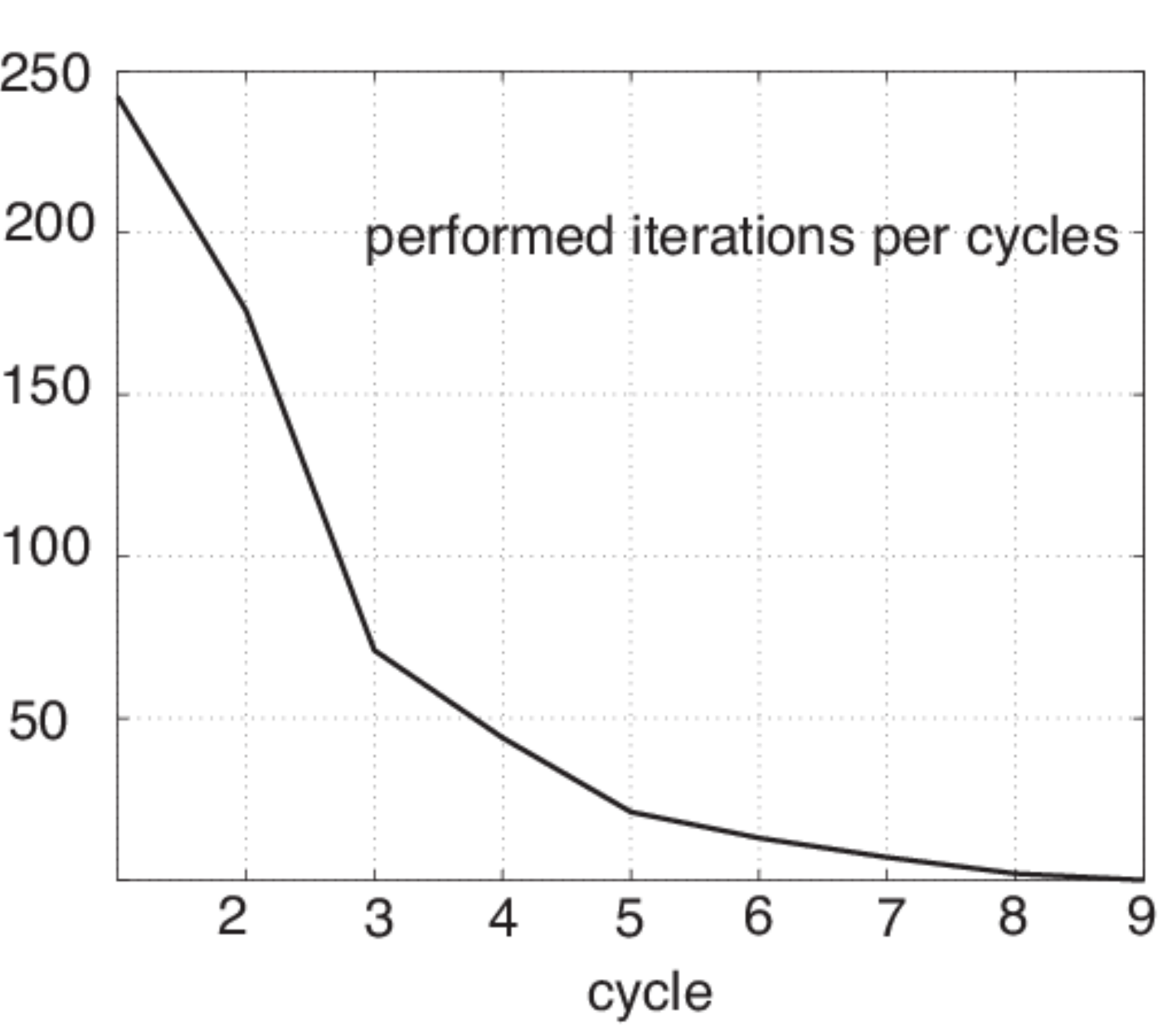}
\end{center}
\caption{
The left image shows the decrease of the error using the \textsc{lLK} method and the
\textsc{LK} method. The right picture shows the actually performed number of iterations
within a cycle of the \textsc{lLK} method.
Within the $9$--th cycle of the \textsc{lLK} method all $\omega_{n} = 0$ and the iteration
terminates.
}\label{sch-steps}
\end{figure}

We note that in all numerical experiments the \textsc{lLK} method applied to Schlieren
tomography converged, even if the tangential cone condition was not satisfied.

\subsection{Discussion}

In contrast to the FPB algorithm the \textsc{lLK} method and the \textsc{LK} method are able to
reconstruct both the positive and negative part of $u$. The \textsc{lLK} method reduces
the artifact which is present in the reconstruction via the \textsc{LK} method.

The left image in Figure \ref{sch-steps} shows that the number of actually computed Landweber
steps per cycle in the \textsc{lLK} method is rapidly decreasing. Moreover, the norm of the error
in the reconstruction with the \textsc{lLK} method is below the error of the \textsc{LK} method.
Actually, the regularized solution  $\x_{n_\ast^\delta}$ of the \textsc{lLK} method is a better
approximation of the true solution than all iterates of the \textsc{LK} method.

Since the used measuring data are squared numbers, there are two solutions, one with
a positive sign and one with a negative sign.
Our numerical simulations showed that a strictly positive initial guess $\x_0$ leads
to a numerical reconstruction $\x_{n_\ast^\delta}$ with positive mean value.

\section{Conclusion}

We applied the abstract theory of the first part of this article
to thermoacoustic computed tomography, to an inverse problem for semiconductors
and to Schlieren tomography. In all applications the \textsc{lLK} method turned 
out to be an efficient iterative regularization method.

\section*{Acknowledgements}

The work of M.H. and O.S. is supported by FWF (Austrian Fonds zur
F\"orderung der wissenschaftlichen Forschung) grants  Y--123INF and P18172.
Moreover, O.S. is supported by FWF projects FSP S9203 and S9207.
The work of A.L. is supported by the Brasilian National Research Council CNPq,
grants 305823/2003--5 and 478099/2004--5.

\bibliographystyle{amsplain}
\bibliography{lit-embedding}

\providecommand{\bysame}{\leavevmode\hbox to3em{\hrulefill}\thinspace}
\providecommand{\MR}{\relax\ifhmode\unskip\space\fi MR }
\providecommand{\MRhref}[2]{%
  \href{http://www.ams.org/mathscinet-getitem?mr=#1}{#2}
}
\providecommand{\href}[2]{#2}
\begin{thebibliography}{10}

\bibitem{Ada75}
R.A. Adams, \emph{Sobolev spaces}, Academic Press, New York, 1975.

\bibitem{AndEtMul00}
V.A. Andreev, A.A. Karabutov, S.~V. Solomatin, E.~V. Savateeva, V.~Aleynikov,
  Y.~V. Zhulina, R.~D. Fleming, and A.~A. Oraevsky, \emph{Opto--acoustic
  tomography of breast cancer with arc--array transducer}, Biomedical
  Optoacoustics (San Jose, CA, USA), vol. 3916, 2000, pp.~36 -- 47.

\bibitem{boc02}
L.~Borcera, \emph{Electrical impedance tomography}, Inverse Problems
  \textbf{18} (2002), no.~6, R99--R136.

\bibitem{Bre98}
M.~A. Breazeale, \emph{Schlieren photography in {P}hysics}, Proc. SPIE
  \textbf{3581} (1998), 41--47.

\bibitem{BurHofPalHalSch05}
P.~Burgholzer, C.~Hofer, G.~Paltauf, M.~Haltmeier, and O.~Scherzer,
  \emph{Thermoacoustic tomography with integrating area and line detectors},
  IEEE Trans. Ultrason. Ferroelec. Freq. Contr. \textbf{52} (2005), 1577 --
  1583.

\bibitem{ChaPel95}
T.~Charlebois and R.~Pelton, \emph{Quantitative 2d and 3d {S}chlieren imaging
  for acoustic power and intensity measurements}, Medical Electronics (1995),
  789--792.

\bibitem{EngHanNeu96}
H.W. Engl, M.~Hanke, and A.~Neubauer, \emph{Regularization of inverse
  problems}, Kluwer Academic Publishers, Dordrecht, 1996.

\bibitem{HLS06}
M.~Haltmeier, A.~Leit\~ao, and O.~Scherzer, \emph{Regularization of systems of
  nonlinear ill-posed equations: I. convergence analysis}, submitted (2006).

\bibitem{HalSchBurNusPal06}
M.~Haltmeier, O.~Scherzer, P.~Burgholzer, R.~Nuster, and G.~Paltauf,
  \emph{Thermoacoustic tomography \& the circular radon transform: Exact
  inversion formula}, submitted (2006).

\bibitem{HanZan94}
A.~Hanafy and C.~I. Zanelli, \emph{Quantitative real-time pulsed {S}chlieren
  imaging of ultrasonic waves}, Proc. IEEE Ultrasonics Symposium \textbf{2}
  (1991), 1223--11227.

\bibitem{KowSch02}
R.~Kowar and O.~Scherzer, \emph{{Convergence analysis of a
  {L}andweber-{K}aczmarz method for solving nonlinear ill-posed problems}}, Ill
  posed and inverse problems (book series) \textbf{23} (2002), 69--90.

\bibitem{KruKisReiKruMil03}
R.A Kruger, W.L. Kiser, D.R. Reinecke, G.A. Kruger, and K.D. Miller,
  \emph{Thermoacoustic molecular imaging of small animals}, Molecular Imaging
  \textbf{2} (2003), no.~2, 113--123.

\bibitem{LedZan99}
E.~G. LeDet and C.~I. Zanelli, \emph{A novel, rapid method to measure the
  effective aperture of array elements}, IEEE Ultrasonics Symposium (1999), --.

\bibitem{LMZ06a}
A.~Leitao, P.A. Markowich, and J.P. Zubelli, \emph{Inverse problems for
  semiconductors: Models and methods}, ch.~in Transport Phenomena and Kinetic
  Theory: Applications to Gases, Semiconductors, Photons, and Biological
  Systems, Ed. C. Cercignani and E. Gabetta, Birkh\"auser, Boston, 2006.

\bibitem{LMZ06}
\bysame, \emph{On inverse dopping profile problems for the stationary
  voltage-current map}, Inv.Probl. \textbf{22} (2006), 1071--1088.

\bibitem{LoiQui00}
A.K. Louis and E.T. Quinto, \emph{Local tomographic methods in sonar}, Surveys
  on solution methods for inverse problems, Springer, Vienna, 2000,
  pp.~147--154.

\bibitem{Nat01}
F.~Natterer, \emph{The mathematics of computerized tomography}, SIAM,
  Philadelphia, 2001.

\bibitem{Nat97}
Frank Natterer, \emph{Algorithms in tomography}, State of the Art in Numerical
  Analysis, vol.~63, 1997, pp.~503--524.

\bibitem{PalNusHalBur06}
G.~Paltauf, R.~Nuster, M.~Haltmeier, and P.~Burgholzer, \emph{Photoacoustic
  tomography using a mach--zehnder interferometer as acoustic line detector},
  Applied Optics (2006), accepted.

\bibitem{PitGreLuKin94}
T.~A. Pitts, J.~F. Greenleaf, Jian-yu Lu, and R.~R. Kinnick, \emph{Tomographic
  {S}chlieren imaging for measurment of beam pressure and intensity}, Proc.
  IEEE Ultrasonics Symposium (1994), 1665--1668.

\bibitem{RamNat35}
C.V. Raman and N.~S. Nath, \emph{The diffraction of light by high frequency
  ultrasonic waves: Part 1}, Proc. Indian Acad. Sci \textbf{2} (1935),
  406--412.

\bibitem{WanPanKuXieStoWan03}
X.D. Wang, G.~Pang, Y.J.~Ku, X.Y. Xie, G.~Stoica, and L.V. Wang,
  \emph{Noninvasive laser-induced photoacoustic tomography for structural and
  functional in vivo imaging of the brain}, Nature Biotechnology \textbf{21}
  (2003), 803--806.

\bibitem{XuMWan03}
M.~Xu and L.V. Wang, \emph{Analytic explanation of spatial resolution related
  to bandwidth and detector aperture size in thermoacoustic or photoacoustic
  reconstruction}, Physical Review E \textbf{67} (2003), 1--15.

\bibitem{XuMWan06}
\bysame, \emph{Photoacoustic imaging in biomedicine}, Review of Scientific
  Instruments \textbf{77} (2006), no.~4, 041101.

\bibitem{XuYEtAl04}
Y.~Xu, L.V. Wang, G.~Ambartsoumian, and Kuchment P., \emph{Reconstructions in
  limited-view thermoacoustic tomography}, Medical Physics \textbf{31} (2004),
  no.~4, 724--733.

\bibitem{ZanKad94}
C.~I. Zanelli and M.~M. Kadri, \emph{Measurements of acoustic pressure in the
  non-linear range in water using quantitative {S}chlieren}, Proc. IEEE
  Ultrasonics Symposium \textbf{3} (1994), 1765--1768.

\end{thebibliography}

\medskip
        {\it E-mail address: }markus.haltmeier@uibk.ac.at\\
 \indent{\it E-mail address: }richard.kowar@uibk.ac.at\\
 \indent{\it E-mail address: }a.leitao@ufsc.br\\
 \indent{\it E-mail address: }otmar.scherzer@uibk.ac.at\\

\end{document}